\documentclass[a4paper,12pt]{article}
\include{epsf}
\usepackage{latexsym}
\usepackage{amssymb}
\usepackage{amsmath}
\usepackage[dvips]{graphicx}
\newlength{\cqfd}
\def\R{\mathbb{R}}

\newtheorem{theorem}{Theorem}

\newtheorem{proposition}{Proposition}

\newtheorem{lemma}{Lemma}

\newtheorem{corollary}{Corollary}

\title{\bf Mathematical derivation of viscous shallow-water equations \\ with zero surface tension}
\author{ Didier Bresch$^{1}$, Pascal Noble$^2$}
\begin{document}

 \maketitle \vspace{-1.5cm}\textit{
\begin{center}
$^1$ Universit\'e de Savoie, \\
Laboratoire de Math\'ematiques, UMR CNRS 5127\\
Campus Scientifique, 73376 Le Bourget du Lac, France\\
{\small email : Didier.Bresch@univ-savoie.fr}\\
$^2$ Universit\'e de Lyon, Universit\'e Lyon 1,\\
UMR CNRS 5208\\ Institut Camille Jordan,
Batiment du Doyen Jean Braconnier,\\
 43, blvd du 11 novembre 1918,
F - 69622 Villeurbanne Cedex, France\\
{\small email : noble@math.univ-lyon1.fr}
\end{center}
}

\noindent
{\bf Abstract.} 
The purpose of this paper is to derive rigorously the so called viscous shallow water equations given for instance page 958-959 in [{\sc A. Oron, S.H. Davis, S.G. Bankoff}, \it Rev. Mod. Phys,  \rm 69 (1997), 931--980]. Such a system of equations is similar to compressible Navier-Stokes equations for a barotropic fluid with a non-constant viscosity. To do that, we consider a layer of incompressible and Newtonian fluid which is relatively thin, assuming {\it no surface tension} at the free surface. The motion of the fluid is described by $3d$ Navier-Stokes equations with constant viscosity and free surface. We prove that for a set of suitable initial data (asymptotically close to ``shallow water initial data''), the Cauchy problem for these equations is well-posed, and the solution converges to the solution of viscous shallow water equations. More precisely, we build the solution of the full problem as a perturbation of the strong solution to the viscous shallow water equations.
 The method of proof is based on a Lagrangian change of variable that fixes the fluid domain and we have to prove the well-posedness in thin domains: we have to pay a special attention to constants in classical Sobolev inequalities and regularity in Stokes problem.

\bigskip



\section{Introduction}
 This paper deals with a free boundary problem of nonstationary Navier-Stokes equations in thin 
domain justifying  the derivation of the viscous shallow water equations without surface tension.
 For information, such a system may be used to describe long-scale evolution of thin liquid films : See Equations (4.12a)-(4.12d) written in \cite{OrDaBa}. Depending on the geometry, there are two kinds of free boundary problems for the motion of a viscous fluid. The first one describes the motion of an isolated mass of fluid bounded by a free boundary and the second one describes the motion of a fluid occupying a semi-infinite domain in $\R^n$ bounded above by a free surface and below by a fixed part of the boundary.  We will consider in the present paper the second problem and refer the interested reader
to \cite{OrDaBa}, \cite{Perthame} for applications for thin liquid films and rivers.
  First mathematical results for free boundary problems were local existence theorems: A first local existence theorem published in 1977 by {\sc V.A.~Solonnikov},  in \cite{Sol1}~; a local existence theorem for the second problem  by {\sc J.T. Beale}, in 1980, in \cite{Beale0}.
   The following years brought other local existence theorems for equations of motion of incompressible
   fluids: see for instance \cite{Al1}, \cite{Al2},  \cite{Ta},  \cite{Ter1}, \cite{Ter2}.
   The next step in investigating free boundary problems for incompressible Navier-Stokes equations was to obtain global existence theorems initiated by {\sc J.T. Beale} in his paper \cite{Beale} published in 1984. This paper was devoted to the motion of a fluid contained in a three dimensional infinite ocean. The first global existence result concerning the motion of a fixed mass of a fluid bounded by a free surface appeared, in the paper \cite{Sol2}, by {\sc V.A. Solonnikov} in 1986.
    The method used to prove the global existence results in the two papers mentioned above are completely different. {\sc J.T.~Beale} examined the surface waves after transformating it to the equilibrium domain through an adequate change of variable (dilatation in the vertical variable) 
 while {\sc V.A. Solonnikov} applied Lagrangian coordinates and this way transformed the considered drop problem to the initial domain. Other global existence theorem can be found. All global existence results for incompressible fluids are obtained for initial data sufficiently close to an equilibrium solution. 
      Note that  {\sc A. Tani}, in \cite{Ta} \cite{TaTa}, used the strategy developped by {\sc V.A. Solonnikov} for free surface problem to get a local existence result through the use of Lagrangian  coordinates to fix the domain and the method of successive approximations.  In contrast to {\sc J.T. Beale}, the regularity of the solution obtained by {\sc A. Tani} is sharp. Note that in {\sc J.T. Beale} existence of arbitrary $T>0$ but for sufficiently small initial data in dependence on $T$ is also proved. The result is obtained in the same class of functions than the local existence.
     To the author's knowledge, in all these studies Dirichlet boundary conditions at the bottom are considered instead of Navier-Slip boundary conditions. Only recent papers with some rigid-lid assumptions consider Navier-Stokes  equations in thin domains with Navier-Slip conditions,
 see for instance  \cite{ChRaRe}, \cite{IfRaSe}.
     In this paper, we will adopt  the Lagrangian transform strategy in order to prove existence and stability of the viscous shallow water solution. To do that, we have to prove the well-posedness in thin domains with a special attention to constants in classical Sobolev inequalities and regularity in Stokes problem. 
    We also prove the stability result in the Lagrangian form as in the paper written by {\sc D. Hoff}, see \cite{Hoff}, where
uniqueness of weak solutions of the Navier-Stokes equations of multidimensional, compressible flow  is
proved using the Lagrangian formulation.
    We remark also that, using the Lagrangian formulation, no surface tension term is needed compared to \cite{Bresch_Noble}. Remark that there exists, at least, two different shallow-water type systems mainly derived from two different bottom  boundary conditions choices for Navier-Stokes equations with free surface:  friction boundary conditions, no-slip boundary conditions. 
   The system based on friction boundary conditions assumption is studied in the present paper and was mentionned in \cite{OrDaBa}. The system based on no-slip assumption has been  formally derived in \cite{BoChNoVi} and recently justified in \cite{Bresch_Noble} assuming non zero surface tension.
   Reader interested by shallow-water equations and related topics introduction is refered to the recent Handbook \cite{Br}. 
   
The paper will be divided in four Sections. The first one presents the Navier-Stokes equations with free surface and the viscous shallow-water approximation. The second section 
is devoted to the linearized Navier-Stokes equations around the approximate solution built from a solution of the viscous
shallow-water equations. We give anisotropic estimates and Korn's inequality in thin domains and then prove the well posedness of the linearized Navier-Stokes equations. The main theorem and its
corollary are proved in Section 4 that means well posedness and convergence of the full Navier-Stokes equations in Lagrangian form. On the last section, we draw some conclusions on this paper and present some perspectives to this work.

\section{Navier-Stokes equations with free surface and viscous shallow-water approximation}
   In this section, we write $3d$ Navier-Stokes equations for a thin layer of an incompressible Newtonian fluid with a free surface. We then show how to obtain an {\it approximate} solution of these equations in the shallow water scaling: for that purpose, we will use smooth solution (see {\sc W. Wang} and {\sc C.J.~Xu}~\cite{WaXu}) of a viscous shallow water model (\ref{sw0}) derived for instance in {\sc A. Oron, S.H.~Davis} and {\sc S.G. Bankoff} \cite{OrDaBa} (see also {\sc J.--F. Gerbeau} and $\>\>$ {\sc B.~Perthame} \cite{Perthame}). We construct formally a second order approximate solution (with respect to the aspect ratio $\varepsilon$ which measures the relative shallowness of the fluid layer) which
will help to solve the Cauchy problem for Navier-Stokes equations with free surface and provides also convergence to viscous shallow-water equations. For this purpose, we formulate the equations in
Lagrangian coordinates.

\subsection{The equations and an existence theorem for the associated viscous shallow-water system}
  In this paper, we consider the motion of a thin layer of Newtonian fluid with a constant density, set to $1$. The differential operators $\nabla_x, {\rm div}_x$ will concern the horizontal coordinates $x=(x_1,x_2)$. In the following, the fluid velocity $u$ is denoted $(u_H,u_V)$ where $u_H\in\mathbb{R}^2$ (resp $u_V\in\mathbb{R}$) is the horizontal (resp. vertical) fluid velocity. Navier-Stokes equations are then written, in their non dimensional form as
{\setlength\arraycolsep{1pt}
\begin{eqnarray}\label{ns_ad}
\displaystyle
\partial_t u+u.\nabla u+\frac{1}{F_0^2}\nabla p&=&-\frac{{\rm e}_3}{F_0^2}+\frac{1}{\rm Re}{\rm div}\big(D(u)\big),\nonumber\\
\displaystyle
{\rm div} \, u&=&0,
\end{eqnarray}}
for all $t>0$ and  $(x,z)\in\Omega_t$ the fluid domain
$$
\displaystyle
\Omega_t=\left\{(x,z)\:/\:x\in\mathbb{X},\quad z\in(0,\,h(x,t))\right\}, \quad \mathbb{X}=\mathbb{T}^n,\mathbb{R}^n,\quad n=1,2.
$$
\noindent
The nondimensional numbers $F_0$, $\rm Re$ are respectively the Froude and Reynolds numbers
defined as 
$$
\displaystyle F_0^2=\frac{U^2}{g\,L},\quad {\rm Re}=\frac{L\,U}{\nu},
$$
\noindent
with $U$ a characteristic velocity of the fluid and $L$ a characteristic wavelength. Navier-Stokes equations are completed by  boundary conditions. The impermeability condition at the free surface reads
\begin{equation}\label{kin0}
\displaystyle
\partial_t h+u_H|_{z=h}.\nabla_x h=u_V|_{z=h},
\end{equation}
\noindent
which means that the fluid velocity is parallel to the free surface. We also assume a Navier slip condition at the bottom
\begin{equation}\label{slip0}
\displaystyle
u_V|_{z=0}=0,\qquad \partial_z u_H|_{z=0}=\gamma u_H|_{z=0},
\end{equation}
and at the free surface, we assume the continuity of the fluid stress
\begin{equation}\label{sliph}
\displaystyle
\big(\frac{1}{\rm Re}D(u)-\frac{p}{F_0^2}\,Id\big)|_{z=h}\,\vec{n}=-\frac{p_{atm}}{F_0^2}\,\vec{n}.
\end{equation}
\noindent
Here $D(u)=\nabla u+\nabla u^T$ is the total deformation tensor of the fluid and can be written  as
$$
\displaystyle
D(u)=\left(\begin{array}{cc} \displaystyle D_x(u_H) & \partial_z u_H+\nabla_x u_V\\
                                                                      \displaystyle (\partial_z u_H+\nabla_x u_V)^T & 2\partial_z u_V
                                        \end{array}\right),\:\:\: D_x(u_H)=\nabla_x u_H+(\nabla_x u_H)^T.
                                        $$                              
\noindent
The vector $\vec{n}$ is the normal to the free surface:
$$
\displaystyle
 \sqrt{1+|\nabla_x h|^2}\,\vec{n}=\left(\begin{array}{c}-\nabla_x h\\1\end{array}\right).
$$
\noindent
Since the fluid pressure is defined up to a constant, we suppose that $p_{atm}=0$. Note that equation (\ref{sliph}) can be equivalently written as
{\setlength\arraycolsep{1pt}
\begin{eqnarray}
\label{sliph_a}
\displaystyle
\frac{p|_{z=h}}{F_0^2}&=&\frac{\big(2\partial_z u_V-(\nabla_x h)^TD_x(u_H)(\nabla_x h)\big)|_{z=h}}{{\rm Re}\big(1-|\nabla_x h|^2\big)},\\
\displaystyle
\partial_z u_H&+&\nabla_x u_V|_{z=h}=\Big(D_x(u_H)-\frac{2\partial_z u_V-(\nabla_x h)^TD_x(u_H)\nabla_x h}{1-|\nabla_x h|^2}\Big)|_{z=h}\nabla_x h.\nonumber
\end{eqnarray}}
\noindent
 In this paper, we will assume that the fluid domain is relatively thin and set $h=\varepsilon\overline{h}$ , $F_0^2=\varepsilon F^2$ with $\varepsilon\ll 1$. As a result, the Froude number $F=\frac{U}{\sqrt{gH}}$ with $H=\varepsilon L$, the characteristic fluid height, is the physical Froude number as the ratio between fluid velocity and wavespeed of perturbations at the free surface. Moreover, following the method of derivation by Gerbeau and Perthame \cite{Perthame} to obtain viscous shallow water equations, we suppose that the friction at the bottom is small: $\gamma=\varepsilon\overline{\gamma}$.  Let us define an approximate solution of Navier-Stokes system (\ref{ns_ad},\ref{kin0},\ref{slip0},\ref{sliph}). For that purpose, we introduce $(h_0,u_0)$ a smooth solution of the shallow water system mentionned by {\sc A. Oron, S.H. Davis, S.G. Bankoff} \cite{OrDaBa}, see page 958-959 with $\overline C \to \infty$:
{\setlength\arraycolsep{1pt}
\begin{eqnarray}
\displaystyle
\partial_t h_0+{\rm div}_x(h_0\,u_0)&=&0, \nonumber\\
\displaystyle
\label{sw0}
h_0\big(\partial_t u_0+u_0.\nabla_x u_0)+\nabla_x\frac{h_0^2}{2F^2}&=&\frac{1}{\rm Re}{\rm div}_x\big(h_0D_x(u_0)\big) \\
\displaystyle
&&+\frac{2}{Re}\nabla_x\big(h_0{\rm div}_x(u_0)\big)-\frac{\overline{\gamma}u_0}{\rm Re}. \nonumber
\end{eqnarray}}
\noindent
   For the moment, we carry out a formal computation: we will make more precise assumptions on this solution in the sequel. We present a second order apporximate solutions, that means we search for an approximate solution of (\ref{ns_ad},\ref{kin0},\ref{slip0},\ref{sliph}) in the form $h_{a}=\varepsilon h_0$ and
$$
\displaystyle
u_{a,H}=u_0+u_1\,z+u_2\,\frac{z^2}{2},\quad u_{a,V}=w_0+w_1\,z+w_2\frac{z^2}{2} + w_3 \frac{z^3}{6},\quad p_{a}=p_0+p_1\,z + p_2 \frac{z^2}{2}.
$$
In the main Theorem we need higher order approximate solution that can be obtained in the same
way, see \cite{Bresch_Noble} for construction of high order approximation of Navier Stokes systems with a free surface.
\noindent
Inserting this ansatz in the divergence free condition, one has 
$$w_1=-{\rm div}_x u_0,\qquad w_2=-{\rm div}_x u_1, \qquad w_3 = -{\rm div}_x u_2.$$
Moreover, the boundary conditions at the bottom imposes
$$w_0=0, \qquad u_1=\varepsilon\overline{\gamma}u_0.$$
 Next, using the second equation of (\ref{ns_ad}) and boundary condition (\ref{sliph_a}), one obtains the approximation of $p$ as
$$
\displaystyle
p_{a}=\varepsilon h_0-z-\frac{2\varepsilon F^2}{\rm Re}{\rm div}_x u_0
               -\frac{2\varepsilon^3 \overline \gamma F^2}{\rm Re}{\rm div}_x u_0.$$
\noindent
Next, we insert $u_{a}$ and $p_{a}$ into the second equation of (\ref{sliph_a}): identifying order $O(\varepsilon)$ terms, one proves that
$$
\displaystyle
u_2+\nabla_x w_1=\big(D_x(u_0)+2{\rm div}_x(u_0)Id\big)\frac{\nabla_x h_0}{h_0}-\frac{\overline{\gamma} u_0}{h_0}.
$$

\noindent {\it Remark.} For the approximate solution, we can calculate an ansatz up to any order 
on derivatives of $(h_0,u_0)$ and polynomial dependence with respect to $z$. We refer to 
\cite{Bresch_Noble} for more details. We end up at order which gives a remaining term of order 
${\cal}(z^3)$ for interior equations and ${\cal}(\varepsilon^3)$ at free surface.

\medskip

\noindent {\it A Navier-Stokes type system is satisfied by the high order shallow water
approximation $(u_a,p_a)$}. The adequate ansatz satisfies
{\setlength\arraycolsep{1pt}
\begin{eqnarray}\label{ns_approx}
\displaystyle
\partial_t u_a+u_a.\nabla u_a+\frac{1}{\varepsilon F^2}\nabla p_a&=&-\frac{{\rm e}_3}{\varepsilon F^2}+\frac{1}{\rm Re}{\rm div}\big(D(u_a)\big)
+ \mathcal{O}(z^3),\nonumber\\
\displaystyle
{\rm div} \, u_a &=&\mathcal{O}(z^3).
\end{eqnarray}}
\noindent
Moreover, the boundary conditions at the bottom are {\it exactly} satisfied and the boundary conditions at the free surface read
$$
\displaystyle
\partial_t h_a+u_{a,H}|_{z=h_a}\nabla_H h_a-u_{a,V}|_{z=h_a}=\mathcal{O}(\varepsilon^3),
$$
\noindent
$$
\displaystyle
\Big(\frac{D(u_a)}{\rm Re}-\frac{p_a}{\varepsilon F^2}\Big)|_{z=h_a}\vec{n}_a=\mathcal{O}(\varepsilon^3).
$$
 For the sake of completeness, we recall the existence and uniqueness result related to the viscous shallow water system (\ref{sw0}) that can be obtained using the same procedure than in \cite{WaXu}
\begin{theorem} Let $s>0$, $u_0(0)$, $h_0(0) -  1 \in H^{s+2}(\mathbb{R}^2)$,
$\|h_0(t=0) - 1\|_{H^{s+2}(\mathbb{R}^2)} \ll 1$.
 Then there exists a positive time $T$, a unique solution $(u_0,h_0)$ of (\ref{sw0}) such that
$$ u_0, h_0- 1  \in L^\infty(0,T; H^{2+s}(\mathbb{R}^2)), \qquad
   \nabla u_0 \in L^2(0,T; H^{2+s}(\mathbb{R}^2)).$$
Furthermore, there exists a constant $c$ such that if
$\|h_0(t= 0) - 1 \|_{H^{s+2}(\mathbb{R}^2)} + \|u_0(t= 0)\|_{H^{s+2}(\mathbb{R}^2)} \le c$ then
we can choose $T= + \infty$ with the same control on $(h_0,u_0)$.
\end{theorem}

\noindent
{\bf Remark:} the smallness assumption on $h_0(t=0)-1$ can be replaced by $h_0(t,x)>\alpha >0$ to prove local well posedness provided that we work in Besov spaces: see the paper of {\sc R. Danchin} \cite{Danchin} for more details. However, we will need here this particular assumption to prove the well posedness of free surface Navier-Stokes equations in dimension $3$ and convergence to shallow water equations. In dimension $2$ for free surface Navier-Stokes equations, we just need an assumption similar to the one of \cite{Danchin}, namely that there is no vacuum. Of course the results are only local in time: for arbitrary time interval, we need a smallness assumption on $h_0(t=0)-1$ and $u_0(t=0)$. 

\medskip
 For global existence of weak solutions to a related viscous shallow water system, a result has been proved in \cite{BrDe} assuming an extra turbulent drag term.

\medskip

\noindent
{\bf Remark:} In this paper, we will choose $s>4$. Indeed, a careful analysis of the remainder shows
that the contained fourth order derivatives of $(h_0,u_0)$ have to be Lipschitz.

\subsection{Reduction of Navier-Stokes equations to a fixed domain}
  In what follows, we write the Navier-Stokes equations in a fixed domain: this is done using a Lagrangian formulation of the equations. We first search the fluid velocity and pressure in the form
$$
\displaystyle
u=u_a+\widetilde u,\quad p=p_a+\frac{\varepsilon F^2}{\rm Re}\widetilde p
$$
\noindent
Navier-Stokes equations (\ref{ns_ad}) then reads
\begin{equation}\label{ns_ap}
\begin{array}{ll}
\displaystyle
\partial_t \widetilde u+(u_a+\widetilde u).\nabla\widetilde u+\widetilde u.\nabla\,{u}_a+\frac{\nabla\widetilde p}{\rm Re}=\frac{1}{\rm Re}{\rm div}\big(D(\widetilde u)\big)+f_a,\\
\displaystyle
{\rm div}\widetilde u=0.
\end{array}
\end{equation}
\noindent
in the fluid domain $\Omega_{\varepsilon,t}=\left\{(x,z)/\,x\in\mathbb{X},\:0\leq z\leq \varepsilon h(x,t)\right\}$. 
The boundary conditions are then written as
\begin{equation} \label{ns_ap1}
\displaystyle
\widetilde u_V|_{z=0}=0,\quad \partial_z \widetilde u_H|_{z=0}=\varepsilon\overline{\gamma}\widetilde u_H|_{z=0}.
\end{equation}
\noindent
At the free surface, one has
\begin{equation} \label{ns_ap2}
\begin{array}{ll}
\displaystyle
\big(D(\widetilde u)-\widetilde p\big)_{z=\varepsilon h}\vec{n}_{\varepsilon}=g_a,\\
\partial_t(\varepsilon h)+(u_{a,H}+\widetilde u_H)|_{z=\varepsilon h}.\nabla_x(\varepsilon h)=u_{a,V}+\widetilde u_V|_{z=\varepsilon h}.
\end{array}
\end{equation}
\noindent
Next, we introduce the Lagrangian change of variable
$$
\displaystyle
\frac{dX}{dt}=u_{a,H}(t,X,Z)+\widetilde u_H(t,X,Z),\quad \frac{dZ}{dt}=u_{a,V}(t,X,Z)+\widetilde u_V(t,X,Z).
$$ 
\noindent
For the sake of simplicity, we assume that the initial fluid domain is $\mathbb{X}\times (0,\,\varepsilon)$: the rigorous justification will remain true for $|<h_0(t=0)>-1|\ll 1$ and $\bar{\gamma}$ sufficienty large. Denote $x_0\in\mathbb{X}, z_0\in(0,\varepsilon)$, the coordinates in this domain and $A$ the Jacobian matrix of the Lagrangian change of variable
$$
\displaystyle
A=\left(\begin{array}{cc}\displaystyle\frac{\partial X}{\partial x_0} & \displaystyle\frac{\partial X}{\partial z_0}\\\displaystyle (\nabla_{x_0} Z)^T &\displaystyle \frac{\partial Z}{\partial z_0}\end{array}\right).
$$
\noindent
The chain rules are given by
$$
\displaystyle
\left(\begin{array}{c}\nabla_{x_0}\\\partial_{z_0}\end{array}\right)=A^T\left(\begin{array}{c}\nabla_{x}\\\partial_{z}\end{array}\right),\quad \left(\begin{array}{cc}\nabla_{x_0}\overline{u}& \partial_{z_0}\overline{u}\\\nabla_{x_0}\overline{w}^T  & \partial_{z_0}\overline{w}\end{array}\right)=\left(\begin{array}{cc}\nabla_{x}\overline{u}& \partial_{z}\overline{u}\\\nabla_{x}\overline{w}^T  & \partial_{z}\overline{w}\end{array}\right)A.
$$
\noindent
The fluid height is defined implicitely as $\varepsilon h\big(t,X(t,x_0,\varepsilon)\big)=Z(t,x_0,\varepsilon)$ so that the impermeability condition is satisfied. Moreover, one has
$$
\displaystyle
\varepsilon\nabla_x h(t,X(t,,x_0,\varepsilon))=(\frac{\partial X}{\partial x_0})^{-T}\nabla_{x_0}Z(t,x_0,\varepsilon).
$$
\noindent
The Lagrangian velocity is defined as $\displaystyle \overline{u}=\widetilde{u}(t,X)$ and $\overline{u}$ defined on the fixed thin domain $\mathbb{X}\times(0, \varepsilon)$. We further introduce the Lagrangian pressure $\overline{p}=\widetilde p(t,X)$. In that setting, Navier-Stokes equations are written as
{\setlength\arraycolsep{1pt}
\begin{eqnarray}
\displaystyle
\partial_t \overline{u}+(A^{-1}\overline{u}).\nabla \overline{u}_a+\frac{A^{-T}\nabla\overline{p}}{{\rm Re}}&=&\frac{A^{-T}}{{\rm Re}}\Big({\rm div}\big(A^T\mathcal{P}\big)-\nabla A:\mathcal{P}\Big)+\overline{f}_a,
\label{Lagr1}\\
\displaystyle
{\rm div}(A^{-1}\overline{u})&=&\overline{\sigma}_a, \label{Lagr2}
\end{eqnarray}} 
\noindent
where $\overline{f}_a=f_a(t,X)$, $\overline{u}_a=u_a(t,X)$, $(\nabla A:\mathcal{P})_i=\sum_{j,k}\partial_k A_{j,i}\mathcal{P}_{j,k}$. The matrix $\mathcal{P}$ is the transformed deformation tensor:
$$
\displaystyle
\mathcal{P}=(\nabla\overline{u})A^{-1}A^{-T}+A^{-T}(\nabla\overline{u})^TA^{-T}.
$$
\noindent
The boundary conditions at the free surface $z_0=\varepsilon$ reads
\begin{equation}\label{Lagr3}
\displaystyle
\Big(\mathcal{P}|_{z_0=\varepsilon}-p|_{z_0=\varepsilon}\,Id\Big){\vec n}=0, \quad \vec{n}=\left(\begin{array}{c} -(\frac{\partial X}{\partial x_0})^{-T}\nabla_{x_0}Z(t,x_0,\varepsilon)\\1\end{array}\right),
\end{equation}
whereas the boundary conditions at the bottom $z=0$ are written as
\begin{equation} \label{Lagr4}
\displaystyle
\overline{u}_{V}|_{z_0}=0,\quad {\rm det}(\frac{\partial X}{\partial x_0})|\partial_{z_0}\overline{u}_H|_{z_0=0}=\varepsilon\overline{\gamma}\,\overline{u}_H|_{z_0=0}.
\end{equation}
The main result of the paper is
\begin{theorem} There exists $C,K>0$ so that if we assume $\|\overline u_0\|_2 \le C^2 \varepsilon^{3/2}$,  then there is
a unique strong solution of \eqref{Lagr1}-\eqref{Lagr4} such that 
$$\sup_{(0,T)} \|\overline u\|^2_2 + \int_0^T \|\overline u\|_3^2 + \|\partial_t \overline u\|_1 + \|\overline p\|_2^2 + \|\partial_t \overline p\|_0^3 Ê<K\varepsilon^{3/2}.$$
\end{theorem}
   This gives the expected existence result of Navier-Stokes equations with free surface namely the system  \eqref{ns_ad}, \eqref{kin0}, \eqref{slip0}, \eqref{sliph_a} and its convergence to the solution of the shallow-water system in a Lagrangian form namely a control of the solution of \eqref{ns_ap}--\eqref{ns_ap2} with respect to
$\varepsilon$.
\begin{corollary} Let $(h_0,u_0)$ be the solution of the viscous shallow-water system \eqref{sw0} such
that $(h_0-1,u_0)$ small enough in ${\cal C}(0,T; H^{2+s}(\R^2)$ with $s>4$. Let $(h_0^\epsilon, u_0^\epsilon)$ such that $\|(u_0^\epsilon-u_a)|_{t=0}\|_2 \le C\varepsilon^{3/2}$ and $h_0^\varepsilon = h_0$ and $u_0^\epsilon$ satisfying the compatibility conditions. Then System  \eqref{ns_ad}, \eqref{kin0}, \eqref{slip0}, \eqref{sliph_a}  is well posed and $u^\epsilon - u_a$ satisfy in its Lagrangian form
the estimates given by Theorem~2.
\end{corollary}
\medskip
Note that there is no assumption on the time interval: the solution of Navier-Stokes system is
defined on the same existence time as the solution of the viscous shallow-water system \eqref{sw0}.

\medskip

\noindent {\it The associated linear system.}
Next we introduce a linear problem that will be usefull to prove the well-posedness of the {\it full} Navier-Stokes system. Indeed, this will be done using a fixed point argument. First, let us introduce the following Lagrangian system of coordinates
$$
\displaystyle
\frac{d X_0}{dt}=u_0(X_0),\quad \frac{d Z_0}{dt}=-Z_0\,{\rm div}\big(u_0\big)(X_0),
$$
\noindent
with the initial conditions $X_0(0,.)=x_0$, $Z_0(0,.)=z_0$. It is easily proved that $X_0,Z_0$ satisfy
$$
\displaystyle
\frac{\partial X_0}{\partial z_0}=0,\quad Z_0(t,x_0,z_0)=z_0 h_0(t,X_0(t,x_0)),\quad {\rm det}\Big(\frac{\partial X_0}{\partial x_0}\Big)=h_0^{-1}(t,X_0(t,x_0)).
$$
\noindent
In what follows, we will consider a special linear problem that is obtained if we substitute $X=X_0$ and $Z=Z_0$ in Navier-Stokes equations written in Lagrangian coordinates and drop order one differential operators. More precisely, this system is written as
\begin{equation}\label{lin_gen_0}
\displaystyle
\partial_t\overline{u}+\frac{A_0^{-T}\nabla p}{\rm Re}=\frac{A_0^{-T}}{\rm Re}{\rm div}(A_0^T\mathcal{P}_0)+f,
\end{equation}
with $\mathcal{P}_0=\nabla(\overline{u})A_0^{-1}A_0^{-T}+A_0^{-T}(\nabla\overline{u})^TA_0^{-T}$. The divergence condition reads
\begin{equation}\label{div_gen_0}
\displaystyle
{\rm div}(A_0^{-1}\overline{u})=\sigma.
\end{equation}
\noindent
The boundary conditions at the bottom are 
\begin{equation}\label{bc_b_0}
\displaystyle
\overline{u}_V|_{z_0=0}=0,\quad h_0^{-1}\partial_{z_0}\overline{u}_H|_{z_0=0}=\varepsilon\overline{\gamma}\,\overline{u}_H|_{z_0=0}+g_1,
\end{equation}
\noindent
whereas the boundary conditions at the free surface read
\begin{equation}\label{bc_s_0}
\displaystyle
\Big(A_0^T\mathcal{P}_0-p\,Id\Big)|_{z_0=\varepsilon}{\rm e}_3=g_2.
\end{equation}

\section{Linearized Navier-Stokes around the approximate solution of viscous shallow-water equations}
 This section is devoted to study the linearized Navier-Stokes equations around the approximate
solution of viscous shallow-water (\ref{lin_gen_0},\ref{div_gen_0},\ref{bc_b_0},\ref{bc_s_0}) that was introduced in the previous section. We begin by a subsection dealing
with Sobolev inequalities, Korn inequality and regularity result for Stokes problem in thin domain.
In the second subsection, we give well posedness result for the full linearized system.

\subsection{Anisotropic estimates in thin domains}

In this part, we recall Sobolev inequalities that arises in thin
domains and pay a special attention to the constants arising in such
inequalities. We also consider a linear Stokes problem in a thin
domain that will be useful to obtain estimates on the Sobolev norms
of $(\overline u,\overline p)$. 

\subsubsection{Sobolev inequalities}

In what follows, we use that mean of $\widetilde u_1$ in the vertical direction is $0$ to write Poincar\'e, Agmon and Ladyzhenskaia anisotropic
inequalities in {\it thin} domains. Here the domain $\Omega$ is
$\Omega=\mathbb{X}^n\times(0,\,\varepsilon)$. In this paper, we will use the following notations: $\|u\|_k=\|u\|_{H^k(\Omega)}$, for any $k\in\mathbb{N}$ with the convention $H^0(\Omega)=L^2(\Omega)$. Similarly, one denotes $|u|_k=|u|_{H^k(\mathbb{X})}$. 
\begin{proposition}\label{int1}
 There exists a constant $C$ independent of $\varepsilon$ such that the following Ladyzhenskaia and Agmon estimates hold true:
$$
\displaystyle
\|u\|_{L^6}\leq C\varepsilon^{-\frac{1}{3}}\|u\|_{1},\quad \|u\|_{L^\infty}\leq C\varepsilon^{-\frac{1}{2}}\|u\|_2.
$$
\end{proposition}
The proof of this proposition is found in \cite{Temam}. From this estimate and the inequality $\displaystyle\|uv\|_{L^2}\leq \varepsilon^{\frac{1}{6}}\|u\|_{L^6}\|v\|_{L^6}$,  one easily proves the following tame estimates for the products of functions and composition of functions.
\begin{lemma}\label{int2}
We will use the following estimates: if $t>\frac{n}{2}$
$$
\begin{array}{ll}
\displaystyle
|u\,v|_{H^s(\mathbb{X})}\leq C\big(|u|_{H^t(\mathbb{X})}|v|_{H^s(\mathbb{X})}+|u|_{H^s(\mathbb{X})}|v|_{H^t(\mathbb{X})}\big),\\
\displaystyle
\|u\,v\|_{H^s(\Omega)}\leq \frac{C}{\sqrt{\varepsilon}}\big(\|u\|_{H^t(\Omega)}\|v\|_{H^s(\Omega)} +\|u\|_{H^s(\Omega)}\|v\|_{H^t(\Omega)}\big),\\
\end{array}
$$
We have also the following property: $\hbox{For all } F: \R^p \rightarrow \R \hbox{ smooth enough with } F(0)=0 \hbox{ then }$
$$ \|F(u)\|_s\leq C(\|u\|_{L^{\infty}})\|u\|_s.
$$
\end{lemma}
\medskip
We will also need the following trace theorems: 
\begin{lemma}\label{trac1}
Let $u\in H^1$ is so that $u|_{z=0}=0$ then one has $|u|_{\frac{1}{2}}\leq C\|u\|_1$, with $C$ independent of $\varepsilon$. In the general case, one has
$\displaystyle|u|_{\frac{1}{2}}\leq \frac{C}{\sqrt{\varepsilon}}\|u\|_1$.
\end{lemma}
\medskip
This result is easily proved with Fourier series, that reduces the problem to a $1d$ problem. The trace of a product is estimated using the following result.
\begin{lemma}\label{trac2}
Let $u\in H^{-\frac{1}{2}}$ and $v\in H^2$ then, one has 
$$
\displaystyle
|uv|_{-\frac{1}{2}}\leq \frac{C}{\sqrt{\varepsilon}}|u|_{-\frac{1}{2}}\|v\|_{2}.
$$
\end{lemma}
The proof of this lemma is found in \cite{CoSh} in a fixed domain: we combine this proof and proposition \ref{int1} to obtain the estimate of lemma \ref{trac2}.

\subsubsection{Korn inequality in thin domain}

In what follows, we prove the following result.
\begin{proposition}\label{prop_korn}
There is a constant $C>0$ independent of $\varepsilon$ such that
$$
\displaystyle
2\,\|D(u)\|^2_{L^2(\Omega)}+\varepsilon\bar{\gamma}|u_{H}|_{z=0}|_{L^2(\mathbb{X})}^2\geq C\|u\|_{H^1(\Omega)}^2,
$$
\noindent
for all $u\in H^1(\Omega)$ so that ${\rm div} u=0$ and $u_V|_{z=0}=0$.
\end{proposition}

\begin{proof}
This is a consequence of Poincar\'e inequality and trace estimates that the $L^2$ norm of $u_H$ and $u_V$ are estimated as
$$
\displaystyle
\|u_V\|_{L^2(\Omega)}\leq C\varepsilon\|\partial_z u_V\|_{L^2(\Omega)},\quad \|u_H\|_{L^2(\Omega)}\leq C\big(\varepsilon|u_H|_{z=0}|_{L^2(\mathbb{X})}+\varepsilon\|\partial_z u_H\|_{L^2(\Omega)}\big).
$$
We will prove a refined inequality on the deformation tensor when $\mathbb{X}=\mathbb{T}^n$ 
$$
\displaystyle 2\,\|D(u)\|^2_{L^2(\Omega)}\geq C\|\nabla u\|_{L^2(\Omega)}^2. 
$$
For that purpose, we decompose the functions into Fourier series in the horizontal variables, $(u_H,u_V)=\sum_{k\in\mathbb{Z}^2}(u_k(z),w_k(z))\exp(ik.x)$. Note that a similar result is obtained when $\mathbb{X}=\mathbb{R}^n$, using Fourier transform. We find, using the divergence free condition 
{\setlength\arraycolsep{1pt}
\begin{eqnarray*}
\displaystyle
2\,\|D(u)\|^2_{L^2(\Omega)}&=&2\sum_{k\in\mathbb{Z}^2}\int_0^{\varepsilon}k_1^2|u_{k,1}|^2+k_2^2|u_{k,2}|^2+|k.u_k|^2\nonumber\\
\displaystyle
+\sum_{k\in\mathbb{Z}^2}\int_0^{\varepsilon}|k_2u_{k,1}+k_1u_{k,2}|^2&+&|u_{k,1}'+\int_0^z\,k_1k.u_k|^2+|u_{k,2}'+k_2\int_0^z\,k.u_k|,\nonumber\\
\displaystyle
\|\nabla u\|_{L^2(\Omega)}^2&=&\sum_{k\in\mathbb{Z}^2}\int_0^{\varepsilon}|k|^2|u_k|^2+|k.u_k|^2+|\int_0^z\,k.u_k|^2+|u_k'|^2.
\end{eqnarray*}}
where $u'(z)=\frac{du}{dz}$. Setting $U_k=\int_0^z u_k$, there exists $C_k\in(0, 2)$ so that
{\setlength\arraycolsep{1pt}
\begin{eqnarray}
\displaystyle
\int_0^{\varepsilon}2k_1^2|U'_{1}|^2&+&2k_2^2|U'_{2}|^2+2|k.U'|+|k_2U'_{1}+k_1U'_{2}|^2+|U_{1}''+k_1k.U|^2\nonumber\\
\displaystyle
+\int_0^{\varepsilon}|U_{2}''&+&k_2\,k.U|^2\geq C_k\int_0^{\varepsilon}|k|^2|U'|^2+|k.U'|^2+|k.U|^2+|U''|^2,
\end{eqnarray}}
\noindent
for all $U\in H^1(0, \varepsilon)^2$ so that $U|_{z=0}=0$. We prove the inequality for each modes. We must prove that there is $C>0$ independent of $k$ so that $C_k\geq C$, forall $k\in\mathbb{Z}^2$. If $k=0$, the two quantities are equal and one can choose $C_0=1$. In what follows, we assume that $k\neq 0$. Denote $C_k\in(0,2)$ the optimal value for the mode $k$. Using test function in $\mathcal{D}(0, \varepsilon)$, one can prove that, if $C_k\neq 1$, an optimal solution necessarily satisfies the differential system
\begin{equation}\label{edo_min}
\displaystyle
D^2 U_1+k_2^2DU_1=k_1k_2 DU_2,\quad D^2U_2+k_1^2DU_2=k_1k_2 DU_1,\quad \forall z\in(0, \varepsilon).
\end{equation}
\noindent
where $D$ is the differential operator $D=\frac{d^2}{dz^2}-|k|^2Id$. This space of solutions is $8$-dimensional and reduced to a $6$-dimensional one, using the condition $U|_{z=0}=0$. Moreover, the eigenfunction $U$ can be written $U(z)=\tilde{U}(|k|z)$, where $\tilde{U}\in\mathcal{S}$, the vector space defined as
{\setlength\arraycolsep{1pt}
\begin{eqnarray}
\displaystyle
\tilde{U}(z)&=&\left(\begin{array}{c} -c\\s\end{array}\right)\big(a_1z+b_1\sinh(z)+c_1(\cosh(z)-1)\big)\nonumber\\
\displaystyle
&&+\left(\begin{array}{c} s\\c\end{array}\right)\big(a_2\sinh(z)+b_2z\sinh(z)+c_2z\cosh(z)\big).
\end{eqnarray}}
where we have denoted $k$ as $k=|k|(c,s)$ and $a_i,b_i,c_i\in\mathbb{R}, i=1,2$. Then we have to show that there is $C$ independent of $k,\varepsilon$ so that $Q_2(\tilde{U},k)\geq C\,Q_1(\tilde{U},k)$, for all $\tilde{U}\in\mathcal{S}$ and $Q_i$ the quadratic forms defined as
{\setlength\arraycolsep{1pt}
\begin{eqnarray*}
\displaystyle
Q_2(U,k)&=&\int_0^M 2c^2(U_1')^2+2s^2(U_2')^2+2(cU_1'+sU_2')^2\\
\displaystyle
&+&\int_0^M(sU_1'+cU_2')^2+\big(U_1''+c(cU_1+sU_2)\big)^2+\big(U_2''+s(cU_1+sU_2)\big)^2,\\
\displaystyle
Q_1(U,k)&=&\int_0^M |U'|^2+(cU_1'+sU_2')^2+|U''|^2+(cU_1+sU_2)^2,
\end{eqnarray*}}
with $M=\varepsilon|k|$. Let us denote $q_i(k), i=1,2$ the symmetric matrix associated to the quadratic forms $Q_i(.,k), i=1,2$. If $C_k$ is the minimum value for which the inequality  $Q_2(\tilde{U},k)\geq C_k\,Q_1(\tilde{U},k)$ for all $\tilde{U}$ is satisfied, then $\det(q_2(k)-C_k q_1(k))=0$ and one is left with an eigenvalue problem. Let us study the roots of $P(X)=\det(q_2(k)-X\,q_1(k))$. First, an integration by parts in the quadratic form $Q_2(.,k)$ yields $Q_2(U,k)=\tilde{Q}_2(U,k)+Q_1(U,k)$
$$
\displaystyle
\tilde{Q}_2(U,k)=2\big(cU_1+sU_2\big)(M)\big(cU_1'+sU_2'\big)(M).
$$
The quadratic form $\tilde{Q}_2(.,k)$ is a product of two linear forms $\tilde{Q}_2(U,k)=(v_1(k),a)(v_2(k),a)$ with $(v_1(k),v_2(k))\in\mathbb{R}^6$ and $a\in\mathbb{R}^6$ represents the coordinates of $\widetilde U$ in the basis of $\mathcal{S}$. Then the matrix $\tilde{q}_2(k)$ associated to $\tilde{Q}_2(.,k)$ is given by $\tilde{q}_2(k)a=(v_1(k),a)v_2(k)+(v_2(k),a)v_1(k)$. As a consequence, for any $a\in (v_1(k),v_2(k))^{\bot}$, one has $q_2(k)a=q_1(k)a$. One can prove that $v_1(k),v_2(k)$ are not colinear so that $X=1$ is a root of $P$ with multiplicity $4$. There remains two extremal values which correspond to minimal and maximal value of $C_k$. We know that $C_k\in(0, 2)$: let us show that $X=2$ is exactly the maximum value for $C_k$. It is sufficient to show that there is a divergence free vector fields $(u,w)$ so that $J(u,w)=J(u,w)^T$ and $w(0)=0$, $J(u,w)$ being the Jacobian matrix of $u,w$. 
This is done choosing a potential flow $(u,w)=(\nabla \psi,\partial_z\psi)$ so that $\psi_z(0)=0$ and $\Delta\psi=0$. Split the functions into their Fourier modes, one shows that it suffices to choose $\psi_k(z)=\alpha_k\cosh(|k|z)$ with $\alpha\in l^2(\mathbb{Z}^2)$ chosen so that the $H^1$ regularity of  $u,w$ is satisfied. As a result, the polynom $P$ is factorized in the form $$P(X)=\Gamma_k(X-1)^4(X-2)(X-\Lambda(k))$$ and there is a unique minimal value $\Lambda(k)$ which depends continuously of  $M=\varepsilon|k|$ and $\sigma=c+is\in\mathbb{S}^1$. We analyse $\Lambda(k)$ when $M\sim\varepsilon$ and $M\to\infty$. We have calculated  with a formal computation software the expansion of $P$ with respect to $M$ at $M=0$: 
$$
\displaystyle
P(X)=M^4\frac{8}{9}(Z-1)^5(Z-2)+\mathcal{O}(M^6),
$$
so that $\Lambda(k)\sim 1$ as $\varepsilon\to 0$ and $k$ lies in a fixed compact set. Next we deal with the limit $M\to\infty$: Assume that there is $U_M$ so that $MQ_1(U_M,k)=1$ and $MQ_2(U_M,k)\to 0$. This is equivalent to the existence of $\tilde{U}_M$ so that 
$$
\begin{array}{ll}
\displaystyle
\int_0^1 M^2(c\tilde{U}_{1,M}+s\tilde{U}_{2,M})^2+\frac{|\tilde{U}_M''|^2}{M^2}+|\tilde{U}_{M}'|^2+(c\tilde{U}_{1,M}'+s\tilde{U}_{2,M}')^2=1,\\
\displaystyle
\lim_{M\to\infty}\big(c\tilde{U}_{1,M}(1)+s\tilde{U}_{2,M}(1)\big)\frac{c\tilde{U}_{1,M}'+s\tilde{U}_{2,M}'}{M}=-1.
\end{array}
$$
Denote $f_M=c\tilde{U}_{1,M}+s\tilde{U}_{2,M}$: it is easily proved that $f_M(1)$ and $f'_M(1)$ are estimated as
$$
\begin{array}{ll}
\displaystyle
f_M(1)^2=2\int_0^1f_M(z)f'_M(z)dz\leq\frac{2}{M},\\
\displaystyle
\frac{1}{M^2}f'_M(1)^2\leq\frac{1}{M^2}\int_0^1(f'_M(z))^2dz+\frac{2}{M^2}\int_0^1\int_z^1|f'_M(s)f''_M(s)|dsdz\leq \frac{3}{M}.
\end{array}
$$
This yields the contradiction with $\lim_{M\to\infty}M^{-1}f_M(1)f'_M(1)=-1$ and concludes the proof of the Korn's inequality.
\end{proof}$\Box$\subsection{Well posedness of linearized Navier-Stokes equations}
\subsubsection{A special divergence free linear problem}

In this section, we consider the special linear problem  obtained from (\ref{lin_gen_0}-\ref{bc_s_0}), assuming $h_0=1$, $u_0=0$: the more general case $|h_0-1|\ll 1$ and $|u_0|\ll 1$ will be treated in the next section through a perturbation argument. In the case considered here, $A_0=Id$ and the linear problem is
\begin{equation}\label{ns_lin_0}
\displaystyle
\partial_t\overline{u}+\frac{\nabla\overline{p}}{\rm Re}=\frac{1}{\rm Re}{\rm div}\,D(\overline{u})+f,\quad {\rm div}u=0,
\end{equation}
\noindent
with the boundary conditions at the bottom 
\begin{equation}\label{bc_b}
\displaystyle
\overline{u}_V|_{z_0=0}=0,\quad \partial_{z_0}\overline{u}_H|_{z_0=0}=\varepsilon\overline{\gamma}\,\overline{u}_H|_{z_0=0}+g_1,
\end{equation}
\noindent
and at the free surface 
\begin{equation}\label{bc_f}
\displaystyle
(\partial_{z_0}\overline{u}_H+\nabla_{x_0}\overline{u}_V)|_{z_0=\varepsilon}=g_{2,1},\quad (2\partial_{z_0}\overline{u}_V-\overline{p})|_{z_0=\varepsilon}=g_{2,2}.
\end{equation}
\noindent
We prove the existence of a weak solution. The weak formulation of (\ref{ns_lin_0},\ref{bc_b},\ref{bc_f}) is written as
{\setlength\arraycolsep{1pt}
\begin{eqnarray}\label{weak_f}
\displaystyle
\partial_t\int_{\Omega}(\overline{u}, \phi)&+&\frac{1}{2\,\rm Re}\int_{\Omega}D(\overline{u}):D(\phi)+\frac{\varepsilon\overline{\gamma}}{\rm Re}\int_{\mathbb{X}}(\overline{u}_H|_{z_0=0}, \phi_H|_{z_0=0})=\nonumber\\
\displaystyle
&&\int_\Omega (f, \phi)+\frac{1}{\rm Re}\int_{\mathbb{X}}(g_2, \phi |_{z_0=\varepsilon})-(g_1, \phi_H|_{z_0=0}),
\end{eqnarray}}
\noindent
for all $\phi\in (H^1(\Omega))^3$ so that ${\rm div}\phi=0$ and $\phi_V|_{z_0=0}=0$. 
\noindent
In what follows, we prove the existence (and uniqueness) of a solution of (\ref{weak_f}).\\

\noindent
\begin{proposition}
Let $f\in L^2\big((0,\,T),(H^1(\Omega)')\big)$ and $g\in L^2\big((0,\,T), H^{-\frac{1}{2}}(\mathbb{T})\big)$. For any $u_0\in L^2_{div}(\Omega)$, there is a unique weak solution $u\in L^2\big((0,\, T), H^1(\Omega)\big)$  of (\ref{weak_f}) and $\partial_t u\in L^2\big((0,\, T), (H^1(\Omega))'\big)$. Moreover, it satisfies the energy estimates
{\setlength\arraycolsep{1pt}
\begin{eqnarray}
\displaystyle
{\rm max}_{(0,\,t)}\,\|u\|_0^2+C\,\int_0^t\,\|u\|_1^2&\leq& \|u_0\|_0^2+\nonumber\\
\label{est_en_0}
\displaystyle
C\int_0^t\|f\|^2_{(H^1(\Omega))'}&+&|g_{2,2}|^2_{-\frac{1}{2}}+\frac{1}{\varepsilon}\big(|g_{2,1}|_{-\frac{1}{2}}^2+|g_1|_{-\frac{1}{2}}^2\big).
\end{eqnarray}}
\end{proposition} 
\noindent
\begin{proof}
The proof follows from a classical Galerkin approximation process: we only prove here the energy estimate. Substituting $\phi=u$ in the weak formulation (\ref{weak_f}), one obtains
{\setlength\arraycolsep{1pt}
\begin{eqnarray}
\displaystyle
\frac{\partial \|u\|_0^2}{\partial t}+\frac{1}{2}\int_{\Omega}D(u)&:&D(u)+\varepsilon\overline{\gamma}\int_{\mathbb{T}}|u_H|_{z=0}|_0^2=\nonumber\\
\label{en_ineq}
\displaystyle
\big(f, u\big)_{(H^1)',H^1}&+&(g_2,u|_{z=\varepsilon})_{H^{-\frac{1}{2}},H^{\frac{1}{2}}}-(g_1,u_H|_{z=0})_{H^{-\frac{1}{2}},H^{\frac{1}{2}}}.
\end{eqnarray}}
\noindent
Using the Korn's inequality on the functional space $\big(H_{div}^1(\Omega)\big)^3$ so that $u_V|_{z=0}=0$, one proves that there exists a constant $C>0$ independent of $\varepsilon$ so that
$$
\displaystyle 
\|u\|_1^2\leq C\Big(\int_{\Omega}D(u):D(u)+\varepsilon\overline{\gamma}\int_{\mathbb{T}}|u_H|_{z=0}|^2\Big).
$$
\noindent
Moreover, using lemma \ref{trac1}, one has the trace estimates
$$
\displaystyle
|u_V|_{z=\varepsilon}|_{\frac{1}{2}}\leq C\|u_V\|_1,\quad \sqrt{\varepsilon}|u_H|_{\frac{1}{2}}\leq C\|u_H\|_1.
$$
\noindent
Using a Young's inequality, one can prove that for any $\eta>0$, there exists $c(\eta)>0$ so that
$$
\begin{array}{lll}
\displaystyle
|(g_{2,1},u_H|_{z=\varepsilon})-(g_1,u_H|_{z=0})_{H^{-\frac{1}{2}},H^{\frac{1}{2}}}|\leq \frac{c(\eta)}{\varepsilon}\big(|g_{2,1}|^2_{-\frac{1}{2}}+|g_1|_{-\frac{1}{2}}^2\big)+\eta\|u_H\|_1^2,\\
\displaystyle
|(g_{2,2},u_V)|_{H^{-\frac{1}{2}},H^{\frac{1}{2}}}|\leq c(\eta)|g_{2,2}|_{-\frac{1}{2}}^2+\eta\|u_V\|_1^2\\
\noindent
|(f, u)|_{(H^1)',H^1}\leq c(\eta)\|f\|_{(H^1)'}^2+\eta\|u\|_1^2.
\end{array}
$$
\noindent
Inserting this inequality into (\ref{en_ineq}) gives the estimate
$$
\displaystyle
\frac{\partial}{\partial t}\|u\|_0^2+C\|u\|_1^2\leq C\Big(\|f\|^2_{(H^1)'}+|g_{2,1}|^2_{-\frac{1}{2}}+\frac{1}{\varepsilon}\big(|g_{2,1}|^2+|g_1|^2\big)\Big).
$$
\noindent
Then integrating this equation on the interval $(0,\,t)$ yields  (\ref{est_en_0}). From the weak formulation, one easily proves
$$
\displaystyle
\int_0^t\|\partial_tu\|^2_{(H^1)'}\leq \|u_0\|_2^2+\int_0^t \|f\|^2_{(H^1)'}+|g_{2,1}|^2_{-\frac{1}{2}}+\frac{1}{\varepsilon}\big(|g_{2,1}|_{-\frac{1}{2}}^2+|g_1|_{-\frac{1}{2}}^2\big)
$$ 
\noindent
Using an interpolation argument, one has $u\in C\big((0,\ t),L^2_{div}\big)$ and the uniqueness follows, letting $f=g_i=0$ and $u_0=0$ into  (\ref{est_en_0}). This completes the proof of the proposition.$\Box$
\end{proof}\\

This is a standard argument that there exists a pressure $p\in L^2((0,\,t), L^2)$ so that $\displaystyle\int_{\Omega} p(.,t)=0$  and
{\setlength\arraycolsep{1pt}
\begin{eqnarray}
\displaystyle
\int_{\Omega} (u_t, \phi)+D(u):D(\phi)&-&\int_{\Omega}p\,{\rm div}\,\phi+\varepsilon\overline{\gamma}\int_{\mathbb{X}}u_H|_{z=0}\,\phi_H|_{z=0}\nonumber\\
\displaystyle
=\int_{\Omega}(f, \phi)+\int_{\mathbb{X}}(g_2, \phi|_{z=\varepsilon})&-&(g_1, \phi_H|_{z=0})\quad\forall\phi\in H^1,\:\phi_V|_{z=0}=0.
\end{eqnarray}}
\noindent
Moreover, the pressure term $p$ satisfies 
$$
\displaystyle
\int_0^t\|p\|^2_{0}\leq C\Big(\|u_0\|_0+\int_0^t\|f\|^2_{(H^1)'}+|g_{2,2}|_{-\frac{1}{2}}^2+\frac{1}{\varepsilon}\big(|g_{2,1}|_{-\frac{1}{2}}^2+|g_1|_{-\frac{1}{2}}^2\big)\Big),
$$
\noindent
where $C$ is a constant independent of $\varepsilon$. Next, we prove a result on the regularity of the weak solution.

\begin{proposition}\label{ex_div0}
Let $f\in L^2\big((0,\,T), H^1\big)$ such that $\partial_t f\in L^2\big((0,\,T), (H^1)'\big)$ and $g_i\in L^2\big((0,\,T), H^{\frac{3}{2}}\big)$, $\partial_t g_i\in L^2\big((0,\,T), H^{-\frac{1}{2}}\big)$ which satisfies the compatibility conditions
$$
\begin{array}{ll}
\displaystyle
(\partial_{z_0} u_H+\nabla_{x_0} u_V)_{z_0=\varepsilon}(t=0,.)=g_{2,1}(0,.),\\
\displaystyle
\big(\partial_{z_0} u_H|_{z_0=0}-\varepsilon\overline{\gamma}u_H|_{z_0=0}\big)(t=0,.)=g_1(0,.).
\end{array}
$$
\noindent
and ${\rm div}\,u_0=0$, $u_0\in H^2$. Then the weak solution $u$ verifies 
{\setlength\arraycolsep{1pt}
\begin{eqnarray*}
\displaystyle
{\rm max}_{(0,\,t)}\big(\|u\|_2^2+\|u_t\|^2_0\big)+\int_0^t\|u(s,.)\|_3^2&+&\|u_t(s,.)\|_1^2ds\leq \nonumber\\
\displaystyle
\|u_0\|_2^2&+&\|f(0,.)\|_0^2+\|g_i(0,.)\|_{\frac{1}{2}}^2\nonumber\\
\displaystyle
+\int_0^t\|f\|_{1}^2+\|f_{t}\|_{(H^1)'}^2+|g_{2,2}|^2_{\frac{3}{2}}&+&|g_{2,2,t}|^2_{-\frac{1}{2}}+\frac{1}{\varepsilon}\big(|g_{i,1}|^2_{\frac{3}{2}}+|g_{i,1,t}|^2_{-\frac{1}{2}}\big)
\end{eqnarray*}} 
\end{proposition}
\begin{proof}
Let us first prove that there exists an ``initial'' pressure $p_0\in H^1(\Omega)$, satisfying the mixed Dirichlet/Neuman problem
$$
\begin{array}{lll}
\displaystyle
-\Delta p_0=-{\rm div}\big(\Delta\,u_0+f(0)\big),\\
\displaystyle
p_0|_{z_0=\varepsilon}=2\partial_{z_0} u_{0,V}|_{z_0=\varepsilon}-g_{2,2}(0),\\ 
\displaystyle
\partial_{z_0} p_0|_{z_0=0}=\varepsilon\overline{\gamma}{\rm div}_{x_0} u_{0,H}|_{z=0}+{\rm div}_{x_0}g_1(0)+f_V(0).
\end{array}
$$
\noindent
This problem is well posed since $f(0)\in L^2$ and $g_i(0)\in H^{\frac{1}{2}}$ with a compactness argument respectively on $f\in L^2(0,\,t;\,H^1)\cap H^1((0,\,t);\,(H^1)')$ and $g\in L^2((0,\,t);\,H^{\frac{3}{2}})\cap H^1(0,\,t;\,H^{-\frac{1}{2}})$. We split this problem into three parts: first, it is easily proved using a Lax Milgram argument that there exists a unique solution $p_{0,1}\in H^1(\Omega)$ so that
$$
\displaystyle
-\Delta p_{0,1}=-{\rm div}\big(\Delta u_0+f(0)\big),\quad p_{0,1}|_{z_0=\varepsilon}=\partial_z p_{0,1}|_{z_0=0}=0.
$$
\noindent
This is a straighforward computation (with Fourier series) to prove that there exists $C$, independent of $\varepsilon$, such that
$$
\displaystyle
\|p_{0,1}\|_1\leq C\big(\|f(0)\|_0+\|u_0\|_2\big).
$$
\noindent
Next we consider the problem
$$
\displaystyle
\Delta p_{0,2}=0,\quad p_{0,2}|_{z=\varepsilon}=h,\quad \partial_{z} p_{0,2}|_{z=0}=0,
$$
\noindent
with $h=2\partial_z u_{0,V}|_{z_0=\varepsilon}-g_{2,2}(0)\in H^{\frac{1}{2}}$. Using Fourier series, this is equivalent to find $p_{0,2}^k$ so that
$$
\displaystyle
-\frac{d^2}{dz^2}p_{0,2}^k+|k|^2 p_{0,2}^k=0,\quad p_{0,2}^k|_{z=\varepsilon}=h_k,\quad \frac{d}{dz}p_{0,2}^k|_{z=0}=0.
$$
\noindent
The solution is easily calculated and satisfies
$$
\displaystyle
p_k(z)=\frac{{\rm cosh}|k|z}{{\rm cosh}(|k|\varepsilon)}h_k,\quad \|\frac{d}{dz}p_{0,2}^k\|^2+|k|^2\|p_{0,2}^k\|\leq C(|k|\varepsilon)|k|h_k^2,
$$
\noindent
with $\displaystyle C(y)=\frac{1}{{\rm cosh}^2y}\int_0^y {\rm cosh}(2x)dx\leq C_1,\forall y\geq 0$ so that $\|p_{0,2}\|_1\leq C_1|h|_{\frac{1}{2}}$.\\ 
\noindent
Similarly, the solution $p_{0,3}$ of 
$$
\displaystyle
\Delta p_{0,3}=0,\:\: p_{0,3}|_{z_0=\varepsilon}=0,\:\: \partial_{z} p_{0,3}|_{z=0}=\varepsilon\overline{\gamma}{\rm div}_{x_0} u_{0,H}|_{z_0=0}+{\rm div}_{x_0} g_1(0)+f_V(0)|_{z_0=0},
$$
\noindent
satisfies  $\displaystyle \|p_{0,3}\|_1\leq C_1\big(\|u_0\|_1+|g_1(0)|_{\frac{1}{2}}+\|f_2(0)\|_0$. Let us define an initial condition $\widetilde u_0\in L^2(\Omega)$ ( that plays the role of $\partial_t u(0)$) as
$$
\displaystyle
\widetilde u_0=-\nabla p_0+\Delta u_0+f(0).
$$
\noindent
This satisfies the following estimate
$$
\|\widetilde u_0\|_0+\|p_0\|_1\leq C_1\Big(\|u_0\|_2+\|f(0)\|_0+|g_1(0),g_2(0)|_{\frac{1}{2}}\Big),
$$
\noindent
with $C_1$ independent of $\varepsilon$. We apply the previous proposition with this initial condition that lies in the space of initial conditions of weak solutions  $\{u\in L^2/\,{\rm div}u=0,\:u_V|_{z_0=0}=0\}$ and $f,g$ are replaced by $(f_t,g_t)$. There exists a unique weak solution $\widetilde u\in L^2(0,\,t;\,H^1(\Omega))$ associated to this problem and that satisfies the energy estimate
{\setlength\arraycolsep{1pt}
\begin{eqnarray}\label{est_u_t}
\displaystyle
{\rm max}_{(0,\,t)}\|\widetilde u\|_0&+&\int_0^t\|\widetilde u\|_1\leq C_1\Big((\|u_0\|^2_2+\|f(0)\|^2_0+|(g_1,g_2)(0)|^2_{\frac{1}{2}}\nonumber\\
\displaystyle
&+&\int_0^t\|f_{t}\|_{(H^1)'}^2+|g_{2,2,t}|_{-\frac{1}{2}}^2+\frac{1}{\varepsilon}\big(|g_{1,t}|_{-\frac{1}{2}}+|g_{2,1,t}|^2_{-\frac{1}{2}}\big).
\end{eqnarray}}
\noindent
Associated to this weak solution is the pressure pressure term $\widetilde p\in L^2$ (which plays the role of $p_t$) that satisfies the estimate $L^2(0,\,t;\,L^2)$
{\setlength\arraycolsep{1pt}
\begin{eqnarray}
\displaystyle
\int_0^t\|\widetilde p\|_{L^2}^2&\leq& C_1\Big((\|u_0\|^2_2+\|f(0)\|^2_0+|g_1(0),g_2(0)|^2_{\frac{1}{2}}\nonumber\\
\displaystyle
&&+\int_0^t\|f_{t}\|_{(H^1)'}^2+|g_{2,2,t}|_{-\frac{1}{2}}^2+\frac{1}{\varepsilon}\big(|g_{1,t}|_{-\frac{1}{2}}+|g_{2,1,t}|^2_{-\frac{1}{2}}\big).
\end{eqnarray}}
\noindent
Then, let us consider $\displaystyle u_1=u_0+\int_0^t\widetilde{u},\quad p_1=p_0+\int_0^t\widetilde{p}.$ Integrating on the time intervals $(0,\,t)$ the weak formulation on $\widetilde{u}$ and using the definition of $p_0$, one can prove that $u_1$ is a weak solution of the initial problem with $u_1(t=0)=u_0$. Then one finds that $u=u_1$, $p=p_1$ and $u_t=\widetilde{u}$, $p_t=\widetilde{p}$ satisfies the estimate (\ref{est_u_t}). This concludes the proof of the proposition.$\Box$
\end{proof}\\

Finally, we check the space regularity of the weak solution. One can prove the proposition
\begin{proposition}\label{prop_reg_lin_0}

Let $f\in L^2(0,\,T;\,H^1)\cap H^1(0,\,T;\,(H^1)')$, $$g_i\in L^2(0,\,T;\,H^{\frac{3}{2}})\cap H^1(0,\,T;\,H^{-\frac{1}{2}})$$ and $u_0\in H^2$ satisfying the compatibility conditions
$$
\displaystyle
\partial_{z_0} u_H+\nabla_{x_0} u_V|_{z_0=\varepsilon}(t=0)=g_{2,1}(0),\quad (\partial_{z_0} u_H-\varepsilon\overline{\gamma}u_H)|_{z_0=0}(t=0)=g_1(0).
$$
\noindent
Then there exists a constant $C$ independent of $\varepsilon$ such that  the weak solution 
$$
\displaystyle
u\in L^2(0,\,T; H^3(\Omega))\cap H^1(0,\,T; H^1(\Omega)),\:p\in L^2(0,\,T; H^2(\Omega))\cap H^1(0,\,T; L^2(\Omega))
$$ 
satisfies
{\setlength\arraycolsep{1pt}
\begin{eqnarray}
\displaystyle
{\rm max}_{(0,\,T)}\|u\|_2^2&+&\int_0^T\|u\|_{3}^2+\|p\|_{2}^2+\|u_t\|_1^2+\|p_t\|^2_0\nonumber\\
\displaystyle
&\leq& C_1\Big(\|u_0\|_2^2+\|f(0)\|_0^2+|(g_1(0),g_2(0))|^2_{\frac{1}{2}}\nonumber\\
\displaystyle
&&+\int_0^T\|f\|_1^2+\|f_{t}\|_{(H^1)'}^2+|g_{2,2}|_{\frac{3}{2}}^2+|g_{2,2,t}|_{-\frac{1}{2}}^2\nonumber\\
\displaystyle
&&+\frac{1}{\varepsilon}\int_0^T\big(|g_1,g_{2,1}|^2_{\frac{3}{2}}+|g_{1,t}, g_{2,1,t}|_{-\frac{1}{2}}^2\big).
\end{eqnarray}}
\end{proposition}
\noindent
\subsubsection{The full linear problem}

In this part, we consider the full problem with a divergence term and $h_0\neq 1$. Let us first deal with the case $h_0=1, u_0=0$ (steady state)  and consider the linear problem (\ref{ns_lin_0},\ref{bc_b},\ref{bc_f}) with the divergence free condition replaced by
\begin{equation}\label{div}
\displaystyle
{\rm div} u=h,
\end{equation}
\noindent 
with $h\in L^2(0,\,T;\,H^2)\cap H^1(0,T;\,L^2)$. In order to prove the well posedness of (\ref{ns_lin_0},\ref{bc_b},\ref{bc_f}) with the divergence condition (\ref{div}), we will apply proposition \ref{ex_div0}: for that purpose, we introduce $u_h=\nabla \phi$ so that
$$
\displaystyle
\Delta\phi=h,\quad \frac{\partial \phi}{\partial z}|_{z=0}=\frac{\partial \phi}{\partial z}|_{z=\varepsilon}=0.
$$ 
\noindent
One can prove, with Fourier series, that there exists a constant $C>0$ independent of $\varepsilon$ such that
$$
\displaystyle
\|\phi\|_4+\|\phi_t\|_2\leq C\big(\|h\|_2+\|h_t\|_0\big).
$$
\noindent
Next, we consider $\widetilde f,\widetilde g_i$ such that $\widetilde{f}=f+\Delta u_h-u_{h,t}$ and
$$
\displaystyle
\widetilde{g}_1=g_1+\varepsilon\overline{\gamma}u_h|_{z=0},\quad \widetilde{g}_{1,2}=g_{1,2}-2\partial_{xz}\phi|_{z=\varepsilon},\quad \widetilde{g}_{2,2}=g_{2,2}+2\partial_{zz}\phi.
$$
\noindent
This is a straightfoward computation to prove that there exists a constant $C$ independent of $\varepsilon$ such that
$$
\begin{array}{llll}
\displaystyle
\|\widetilde{f}\|_1\leq C\big(\|f\|_1+\|h\|_2+\|h_t\|_0\big),\\
\\
\displaystyle
|\widetilde{g}_1,\widetilde{g}_{1,2}|_{\frac{3}{2}}+|\widetilde{g}_{1,t},\widetilde{g}_{1,2,t}|_{-\frac{1}{2}}\leq C\big(|g_1,g_{1,2}|_{\frac{3}{2}}+|g_{1,t},g_{1,2,t}|^2_{-\frac{1}{2}}+\|h\|_2+\|h_t\|_0\big),\\
\displaystyle
|\widetilde{g}_{2,2}|_{\frac{3}{2}}+|\widetilde{g}_{2,2,t}|_{-\frac{1}{2}}\leq C\Big(|g_{2,2}|_{\frac{3}{2}}+|g_{2,2,t}|_{-\frac{1}{2}}+\frac{1}{\sqrt{\varepsilon}}\big(\|h\|_2+\|h_t\|_0\big)\Big).
\end{array}
$$
\noindent
The estimate on $\|\widetilde{f}_{t}\|_{(H^1)'}$ is obtained using the linear Navier-Stokes equations: more precisely, one has $\displaystyle\widetilde{f}_t=f_t+\Delta u_{h,t}-u_{h,tt}$ and 
$$
\displaystyle
{\rm div}{u}_{h,tt}=h_{tt}={\rm div}u_{tt}={\rm div}\Big(f_t+\frac{1}{Re}{\rm div}\big(D(u_t)\big)-\frac{1}{Re}\nabla p_t)\Big). 
$$
\noindent
Then, one obtains the estimate of $u_{h,tt}$ (and thus $\widetilde{f}_t$) in $L^2((0, T), (H^1)')$ norm, using elliptic regularity. Next, we apply proposition \ref{ex_div0} on the divergence free problem for the fluid velocity $u-u_h$ and pressure $p$. Finally, we have proved the well-posedness of the linear problem (\ref{ns_lin_0},\ref{bc_b},\ref{bc_f}) with (\ref{div}) and the solution $(u,p)$ satisfies the energy estimates
{\setlength\arraycolsep{1pt}
\begin{eqnarray}
\displaystyle
{\rm max}_{(0,\, T)}\big(\|u\|^2_2&+&\|u_t\|^2_0\big)+\int_0^T\|u\|_3^2+\|u_t\|_1^2+\|p\|_2^2+\|p_t\|_0^2\nonumber\\
\displaystyle
&\leq&C\Big(\|u_0\|_2+\|f(0)\|^2_{0}+|(g_1,g_2)(0)|_{\frac{1}{2}}^2\nonumber\\
\displaystyle
&+&\int_0^T\|f\|_1^2+\|f_{t}\|^2_{(H^1)'}+|g_{2,2}|_{\frac{3}{2}}^2+|g_{2,2,t}|^2_{-\frac{1}{2}}\nonumber\\
\displaystyle
\label{en_fin}
&+&\frac{1}{\varepsilon}\int_0^T\big(\|h\|_2^2+\|h_t\|_0^2+|g_1,g_{2,1}|_{\frac{3}{2}}^2+|g_{1,t},g_{2,1,t}|_{-\frac{1}{2}}^2\big).
\end{eqnarray}}

\noindent
Let us now deal with the full problem (\ref{lin_gen_0},\ref{div_gen_0},\ref{bc_b_0},\ref{bc_s_0}). For the sake of simplicity, we assume $\sigma=0$, the inhomogeneous case $\sigma\neq 0$ is treated as previously. In order to deal with the more classical divergence free condition, we introduce the Jacobian decomposition of the fluid velocity $\overline{u}=A_0\widetilde{u}$. The linear problem then reads (we also drop a linear term in the form $L_0(u,\nabla u,x,t)$)
\begin{equation}\label{lin_gen_1}
\displaystyle
A_0^T\partial_t\big(A_0\widetilde{u}\big)+\frac{1}{\rm Re}\nabla p=\frac{1}{\rm Re}{\rm div}(\widetilde{\mathcal{P}}_0)+f,
\end{equation}
with $\widetilde{\mathcal{P}}_0=A_0^TA_0\nabla(\widetilde{u})A_0^{-1}A_0^{-T}+(\nabla\widetilde{u})^T$. The divergence condition reads
\begin{equation}\label{div_gen_1}
\displaystyle
{\rm div}\widetilde{u}=0.
\end{equation}
\noindent
The boundary conditions at the bottom are written as
\begin{equation}\label{bc_b_1}
\displaystyle
\widetilde{u}_V|_{z_0=0}=0,\quad h_0^{-1}\partial_{z_0}\widetilde{u}_H|_{z_0=0}=\varepsilon\overline{\gamma}\,\widetilde{u}_H|_{z_0=0}+g_1,
\end{equation}
\noindent
whereas the boundary conditions at the free surface read
\begin{equation}\label{bc_s_1}
\displaystyle
\Big(\widetilde{\mathcal{P}}_0-p\,Id\Big)|_{z_0=\varepsilon}{\rm e}_3=g_2.
\end{equation}
Similarly to the special linear problem treated in the previous section, one can obtain a weak formulation of (\ref{lin_gen_1}-\ref{bc_s_1}):
$$
\begin{array}{ll}
\displaystyle
\partial_t\int_{\Omega}\big(A_0\widetilde{u},\,A_0\phi\big)+\frac{1}{Re}\int_{\Omega}\nabla\phi:\widetilde{\mathcal{P}}_0+\frac{\varepsilon\overline{\gamma}}{Re}\int_{\mathbb{X}}h_0^{-1}(\frac{\partial X_0}{\partial x_0}\widetilde{u}_H|_{z_0=0},\frac{\partial X_0}{\partial x_0}\phi_H|_{z_0=0})\\
\displaystyle
=\int_{\Omega}(A_{0,t}\phi, A_0\widetilde{u})+(f,\,\phi)+\frac{1}{Re}\int_{\mathbb{X}}(g_2,\,\phi_H|_{z_0=\varepsilon})-h_0^{-1}(\frac{\partial X_0}{\partial x_0}g_1,\,\frac{\partial X_0}{\partial x_0}\phi_H|_{z=0}). 
\end{array}
$$
\noindent
for all $\phi\in H^1(\Omega)$ so that ${\rm div}\phi=0$ and $\phi_V|_{z_0=0}$. One prove the following theorem

\begin{theorem}\label{th_ex_h0_not_1}
There is $\eta_1>0,\eta_2>0$ such that for any $(h_0, u_0)\in W^{2,\infty}$ be a smooth solution of (\ref{sw0}) which first satisfies $|h_0(t=0)-1|_{\infty}<\eta_1$ and either
\begin{itemize}
\item (small data, large time) $\forall t\in(0, T), \|h_0(t,.)-1,u_0(t,.)\|_{\infty}\leq \eta_2$ or,
\item (large data, small time) $T\sup_{t\in(0, T)}\|h_0(t,.)-1,u_0(t,.)\|_{W^{2,\infty}}\leq \eta_2$
\end{itemize}

\noindent
Then for all source terms $f,g_i,\;i=1,2$ and $h$ so that
$$
\begin{array}{ll}
\displaystyle
f\in L^2((0,\,T);\,H^1(\Omega))\cap H^1((0,\,T);\,(H^1)'(\Omega)),\\
\displaystyle
g_i\in L^2((0,\,T);\,H^{\frac{3}{2}}(\mathbb{X}))\cap H^1((0,\,T);\,H^{-\frac{1}{2}}(\mathbb{X})),\\
\displaystyle
h\in L^2((0,\,T);\,H^2(\Omega))\cap H^1((0,\,T);\,L^{2}(\Omega)).
\end{array}
$$
\noindent
There exists a unique weak solution $(u,p)$ of (\ref{ns_lin_0},\ref{bc_b},\ref{bc_f},\ref{div}) and this solution satisfies energy estimates (\ref{en_fin}) .
\end{theorem}


\begin{proof}
The proof of this theorem follows the one of proposition \ref{ex_div0}: one first constructs an approximate solution using a Galerkin method and then prove uniform estimates on this solution. The only difference here is that we have to prove the Korn's inequality when $(h_0,u_0)\approx(1,0)$. Let us first assume that $h_0(t=0)=1$. In the case of small data, one can write that 
$$
\begin{array}{ll}
\displaystyle
\int_{\Omega}\nabla\tilde{u}:\tilde{\mathcal{P}}_0=\frac{1}{2}\int_{\Omega}D(\widetilde{u}):D(\widetilde{u})+\mathcal{O}(\eta_2)\|\nabla\widetilde{u}\|^2,\\
\displaystyle
\int_{\mathbb{X}}h_0^{-1}|\frac{\partial X_0}{\partial x_0}\widetilde{u}_H|_{z_0=0}|^2=\big(1+\mathcal{O}(\eta_2)\big)\int_{\mathbb{W}}|\widetilde{u}_H|_{z_0=0}|^2
\end{array}
$$
Then, one can use the Korn's inequality proved in the case $h_0=1$ to show that for $\eta_2$ sufficiently small, there is $C>0$ independent of $\varepsilon$ so that
$$
\displaystyle
\int_{\Omega}\nabla\tilde{u}:\tilde{\mathcal{P}}_0+\varepsilon\overline{\gamma}\int_{\mathbb{X}}h_0^{-1}|\frac{\partial X_0}{\partial x_0}\widetilde{u}_H|_{z_0=0}|^2\geq C\|\nabla\widetilde{u}\|^2.
$$
\noindent
From the Korn inequality, we deduce energy estimates on the approximate solutions: this proves the existence of a weak solution. The proof of regularity is similar to the case of $h_0=1, u_0=0$: indeed the main task was to prove a Korn's inequality in this setting, which is done mainly using the fact that the Lagrangian change of variable is close to identity uniformly with respect to $\varepsilon$. The case of large data and small time is simply treated by considering that the Lagrangian change of variable is close to identity for $T$ small enough and one can prove a similar Korn's inequality in this case. The more general case $h_0(t=0)\neq 1$ is treated by perturbation, simply through the change of variable $(x,z)\mapsto \big(x, zh_0(t=0,x)\big)$. $\Box$
\end{proof}

\section{Well posedness and convergence of the full Navier-Stokes equations in Lagrangian form.}
 We write an iterative scheme to prove the well posedness of free surface Navier Stokes equations with Navier slip condition at the bottom together with the convergence of this set of equations to a viscous shallow water model.  We will use Tychonov fixed point theorem, as in \cite{CoSh}, with now the functional space
$$
\begin{array}{ll}
\displaystyle
\mathbb{X}=\big\{(u,p)\in L^2\big((0, T); H^3\times H^2\big)\;/\; \|(u,p)\|_{\mathbb{X}}\leq C(\|u_0\|_2
  + \varepsilon^{3/2})
    \big\},\\
\displaystyle
\|(u,p)\|_{\mathbb{X}}^2=\sup_{(0, T)}\|u\|^2_2+\int_0^T \|u\|^2_3+\|\partial_t u\|_1^2+\|p\|_2^2+\|\partial_t p\|_0^2.
\end{array}
$$
\noindent
The iterative scheme is written as
\begin{equation}\label{mom_k}
\begin{array}{ll}
\displaystyle
\partial_t u^{k+1}+A_0^{-1}u^{k+1}.\nabla u_a(X_0)+\frac{A_0^{-T}\nabla p^{k+1}}{R_e}
\nonumber \\
 \displaystyle 
 \hskip5cm =\frac{A_0^{-T}}{R_e}\Big({\rm div}\big(A_0^T\mathcal{P}_0^{k+1}\big)-\nabla A_0:\mathcal{P}_0^{k+1}\Big)+F^k, \\
\displaystyle
{\rm div}\big(A_0^{-1}u^{k+1}\big)= H_k,
\end{array}
\end{equation}
with 
$$H_k=\sigma_a(X_0,Z_0) + (\sigma_a(X^k,Z^k)
            - \sigma_a(X_0,Z_0)) +{\rm div}\big((A_0^{-1}-A_k^{-1})u^k\big)$$
and
{\setlength\arraycolsep{1pt}
\begin{eqnarray}
\displaystyle
F^k&=&(A_0^{-1}-A_k^{-1})u^k.\nabla u_a(X^k,Z^k)+A_0^{-1}u^k.\nabla\big(u_a(X_0,Z_0)-u_a(X^k,Z^k)\big)\nonumber\\
\displaystyle
&+&\frac{1}{R_e}\Big((A_0^{-T}-A_k^{-T})\nabla p^k+(A_k^{-T}-A_0^{-T})\big({\rm div}(A_k^T\mathcal{P}^k)-\nabla A_k:\mathcal{P}^k\big)\Big)\nonumber\\
\displaystyle
&+&\frac{1}{R_e}A_0^{-T}\Big(\nabla A_0:\mathcal{P}_0^k-\nabla A_k:\mathcal{P}^k+{\rm div}\big({A_k}^T\mathcal{P}^k-A_0^T\mathcal{P}_0^k\big)\Big) \nonumber\\
& & \hskip4cm + (f_a(X^k,Z^k)- f_a(X_0,Z_0)) + f_a(X_0,Z_0).\nonumber
\end{eqnarray}}
This set of equation is completed with boundary conditions at the bottom:
\begin{equation}
\displaystyle
u^{k+1}_V|_{z=0}=0, 
\frac{1}{h_0}\partial_z u_H^{k+1}|_{z=0}=\varepsilon\overline{\gamma}u_H^{k+1}|_{z=0}
     + G_k^1 ,
\end{equation}
with $G_k^1= \varepsilon\overline{\gamma}\Big(\big(\frac{{\rm det}(\partial_x X_0)}{{\rm det}(\partial_x X^k)}-1\big)u_H^k\Big)|_{z=0}$
\noindent
and boundary conditions at the free surface
\begin{equation}
\displaystyle
\big(A_0^T\mathcal{P}_0^{k+1}-p^{k+1}Id\big)|_{z=\varepsilon}{\rm e}_3= G_k^2 ,
\end{equation}
where 
$$G_k^2 = A_0^T\Big((g_a(X^k,Z^k) - g_a(X_0,Z_0)) + g_a(X_0,Z_0) 
    +(\mathcal{P}^k_0-\mathcal{P}^k){\rm n}+(\mathcal{P}_0^k-p^k Id)({\rm n}_0-{\rm n})\Big)|_{z=\varepsilon}
$$
with 
$$
\begin{array}{ll}
\displaystyle
\mathcal{P}^k=(\nabla u^k)A_k^{-1}A_k^{-T}+A_k^{-T}(\nabla u^k)^TA_k^{-T},\\
\displaystyle
\mathcal{P}_0^k=(\nabla u^k)A_0^{-1}A_0^{-T}+A_0^{-T}(\nabla u^k)^TA_0^{-T}.
\end{array}
$$
The function $u^{k+1}$ has the initial condition $u^{k+1}|_{t=0}=u_0$. 
   Using the well posedness results obtained in the previous section, this set of equation has a unique solution and the weak continuity of the map is done. We first prove that $(u^k,p^k)$ is bounded, uniformly on the time interval $(0, T)$ and uniformly with respect to $\varepsilon$, in a suitable norm provided that the initial condition is sufficiently small. The proof is done in $5$ steps: estimates of the Lagrangian coordinates (more precisely the deviation from the ``shallow water'' coordinates), the boundary terms, the divergence term and nonlinear terms in momentum equations.\\

\noindent
{\it Step 1: Estimates of Lagrangian coordinates.} Let us recall that  Lagrangian coordinates satisfy the set of equations
$$
\begin{array}{ll}
\displaystyle
\frac{d}{dt}(X^k-X_0)=u_{a,H}(X^k,Z^k)-u_{a,H}(X_0,Z_0)+u_H^k,\\
\displaystyle
\frac{d}{dt}(Z^k-Z_0)=u_{a,V}(X^k,Z^k)-u_{a,V}(X_0,Z_0)+u_V^k. 
\end{array}
$$
Let us first compute the basic $L^2$ estimate
{\setlength\arraycolsep{1pt}
\begin{eqnarray}
\displaystyle
\frac{d}{dt}\big(\|X^k-X_0\|_0^2+\|Z^k-Z_0\|_0^2\big)&\leq& C(\|u_a(t,.)\|_{W^{1,\infty}})\big(\|X^k-X_0\|_0^2+\|Z^k-Z_0\|_0^2\big)\nonumber\\
\displaystyle
&&+\|u^k\|_0\big(\|X^k-X_0\|_0+\|Z^k-Z_0\|_0\big).\nonumber
\end{eqnarray}}
\noindent
This is a straightforward consequence of Gronwall inequality that
$$
\displaystyle
\sup_{t\in(0, T)}\big(\|X^k-X_0\|_0^2+\|Z^k-Z_0\|_0^2\big)\leq Te^{C(T+1)}\int_0^T\|u^k(t,.)\|_0^2\,dt,
$$
\noindent
with $C=C\big(\sup_{t\in(0, T)}(\|u_a\|_{W^{1,\infty}})\big)$. There is no difficulty to extend this estimate to higher order space derivatives:
$$
\displaystyle
\sup_{t\in(0, T)}\big(\|X^k-X_0\|_{s}^2+\|Z^k-Z_0\|^2_s\big)\leq Te^{C(T+1)}\int_0^T\|u^k(t,.)\|_s^2\,dt,
$$
\noindent
with $C=C\big(\sup_{t\in(0, T)}(\|u_a\|_{W^{1,\infty}})\big)$. Next, we compute boundary terms at the bottom:\\

\noindent {\it Step 2: Estimates of nonlinear terms at the bottom.}
In order to estimate the boundary term at the bottom, we will use lemma \ref{trac1} and \ref{trac2}. The nonlinear term at the bottom reads
$$
\displaystyle
G_1^k=\varepsilon\overline{\gamma}u^k_H|_{z=0}\big(\frac{{\rm det}(\partial_x X_0)}{{\rm det}(\partial_x X^k)}-1\big)|_{z=0}=\varepsilon\overline{\gamma}u^k_H|_{z=0}\mathcal{F}(\partial_x X^k-\partial_x X_0).
$$
\noindent
We first compute an estimate for $|g_1^k|_{\frac{3}{2}}$. Let us fix $t_0>1$ so as $H^{t_0}(\mathbb{T}^2)\subset L^{\infty}(\mathbb{T}^2)$. We then have
{\setlength\arraycolsep{1pt}
\begin{eqnarray}
\displaystyle
|G_1^k|_{\frac{3}{2}}&\leq&\varepsilon\overline{\gamma}\Big(|u^k_H|_{z=0}|_{t_0}|\mathcal{F}(\partial_x X^k-\partial_x X_0)|_{\frac{3}{2}}+|u^k_H|_{z=0}|_{\frac{3}{2}}|\mathcal{F}(\partial_x X^k-\partial_x X_0)|_{t_0}\Big)\nonumber\\
\displaystyle
&\leq&\varepsilon\overline{\gamma}C(|\partial_x X^k-\partial_x X_0|_{\infty})\big(|u^k_H|_{z=0}|_{t_0}|\partial_x X^k-\partial_x X_0|_{t_0} \nonumber \\
& & \hskip6cm +|u_H^k|{z=0}|_{\frac{3}{2}}|\partial_x X^k-\partial_x X_0|_{t_0}\big) \nonumber\\
&\leq&\varepsilon\overline{\gamma}C(|X^k-X_0|_{t_0+1})\Big(\frac{\|u^k_H\|_{t_0+\frac{1}{2}}}{\sqrt{\varepsilon}}\frac{\|X^k-X_0\|_3}{\sqrt{\varepsilon}}+\frac{\|u_H^k\|_{2}}{\sqrt{\varepsilon}}\frac{\|X^k-X_0\|_{t_0+\frac{3}{2}}}{\sqrt{\varepsilon}}\Big)\nonumber\\
&\leq&\overline{\gamma}C(\frac{\|X^k-X_0\|_{t_0+\frac{3}{2}}}{\sqrt{\varepsilon}})\big(\|u^k\|_{t_0+\frac{1}{2}}\|X^k-X_0\|_3 \nonumber\\
& & \hskip6cm +\|u^k\|_2\|X^k-X_0\|_{t_0+\frac{3}{2}}\big).
\end{eqnarray}}
\noindent
Then it is easily seen, choosing $t_0=\frac{3}{2}$, that
$$
\displaystyle
\int_0^T|G^k_1(t,.)|_{\frac{3}{2}}^2 dt\leq\overline{\gamma}C\big(\sqrt{\frac{T}{\varepsilon}}\|u^k\|_{L^2((0, T), H^3)}\big)T\sup_{t\in(0, T)}(\|u^k(t,.)\|_2^2)\int_0^T\|u^k(t,.)\|_3^2 dt.
$$
\medskip
Using lemma \ref{trac2}, we next compute an estimate of $|\partial_t g_1^k|_{-\frac{1}{2}}$. One has
$$
\displaystyle
\partial_t G_1^k=\varepsilon\overline{\gamma}\Big(\partial_t u^k_H|_{z=0}\mathcal{F}(\partial_x X^k-\partial_x X_0)+u_H|_{z=0}\partial_1\mathcal{F}(\partial_x X^k-\partial_x X_0)\big(\partial_{xt}^2(X^k-X_0)\big)\Big)
$$
\noindent
Then, the following inequality holds:
{\setlength\arraycolsep{1pt}
\begin{eqnarray}
\displaystyle
|\partial_t G_1^k|_{-\frac{1}{2}}&\leq&\varepsilon\Big(\frac{|\partial_t u^k|_{-\frac{1}{2}}}{\sqrt{\varepsilon}}\|\mathcal{F}(\partial_x X^k-\partial_x X_0)\|_2 \nonumber \\
& & \hskip2cm +|u_H|_{z=0}\mathcal{G}(\partial_x X^k-\partial_x X_0)(\partial_x u_H^k+\mathcal{F}_1(\partial_x X^k-\partial_x X_0)|_{-\frac{1}{2}}\Big)\nonumber\\
\displaystyle
&\leq&C\big(\sqrt{\frac{T}{\varepsilon}}\|u^k\|_{L^2((0, T), H^3)}\big)(\|u^k\|_1+\|\partial_t u^k\|_1)\|X^k-X_0\|_3 \nonumber\\
& & \hskip4cm +\sqrt{\varepsilon}|\partial_x u^k|_{-\frac{1}{2}}\|\mathcal{F}(\partial_x X^k-\partial_x X_0)u_H\|_{2}\\
&\leq&C\big(\sqrt{\frac{T}{\varepsilon}}\|u^k\|_{L^2((0, T), H^3)}\big)\Big((\|u^k\|_1
 +\|\partial_t u^k\|_1)\|X^k-X_0\|_3+\|u^k\|_2\|X^k-X_0\|_3\Big).\nonumber
\end{eqnarray}}
So that we obtain the estimate
$$
\displaystyle
\int_0^T|\partial_t G_1^k|_{-\frac{1}{2}}^2\leq C\big(\sqrt{\frac{T}{\varepsilon}}\|u^k\|_{L^2((0, T), H^3)}\big)T\big(\int_0^T\|\partial_t u^k\|_1^2+\|u^k\|_3^2\big)^2.
$$
\medskip
Next we compute interior estimates of  ``divergence'' terms: this is done using proposition \ref{int1} and tame estimates of lemma \ref{int2}.
\medskip

\noindent {\it Step 3: Estimate of divergence terms.} 
 One part of the divergence term $H_1^k={\rm Tr}\big((A_k^{-1}-A_0^{-1})\nabla u^k\big)$
 satisfies 
 $$
\begin{array}{ll}
\displaystyle
\|H_1^k\|_2\leq \frac{1}{\sqrt{\varepsilon}}C(\|\partial_x X^k-\partial_x X_0\|_\infty)\|X^k-X_0\|_3\|\nabla u^k\|_2,\\
\displaystyle
\|H^k_1\|_2\leq \frac{1}{\sqrt{\varepsilon}}C(\sqrt{\frac{T}{\varepsilon}}\|u^k\|_{L^2((0, T), H^3)})\|X^k-X_0\|_3\|u^k\|_3.
\end{array}
$$
\noindent
As a consequence, one has
$$
\displaystyle
\int_0^T\|H_1^k(t,.)\|_2^2 dt\leq \frac{C(\sqrt{\frac{T}{\varepsilon}}\|u^k\|_{L^2((0, T), H^3)})}{\varepsilon}\big(\int_0^T\|u^k\|_3^2\big)^2.
$$
\noindent
Next, we compute an estimate of $\|\partial_t H^k_1\|_0$. On the one hand, one has
$$
\displaystyle
\|{\rm Tr}\big((A_k^{-1}-A_0^{-1})\nabla \partial_t u^k\big)\|_0\leq\frac{1}{\sqrt{\varepsilon}}C(|\partial_x X^k-\partial_x X_0|_{\infty})\|X^k-X_0\|_3\|\partial_t u^k\|_1, 
$$
\noindent
and
$$
\begin{array}{ll}
\displaystyle
\|{\rm Tr}\big((\partial_t A_k-\partial_t A_0)\nabla u^k\big)\|_0\leq \|\nabla u^k\|_{\infty}\|C(\partial_{x,t}X^k-\partial_{x,t}X_0)\|_0,\\
\displaystyle
\|{\rm Tr}\big((\partial_t A_k-\partial_t A_0)\nabla u^k\big)\|_0\leq\frac{1}{\sqrt{\varepsilon}}\frac{\|u^k\|_3}{\sqrt{\varepsilon}}C(|\partial_x X^k-\partial_x X_0|_\infty \nonumber \\
\hskip8cm +|\partial_x u^k|_\infty)\big(\|X^k-X_0\|_1+\|u^k\|_1\big).
\end{array}
$$
\noindent
So that we obtain
$$
\displaystyle
\int_0^T\|\partial_t H_1^k\|_0^2\leq \frac{T}{\varepsilon}C\big(\sqrt{\frac{T}{\varepsilon}}\|u^k\|_{L^2((0, T), H^3)}\big)\int_0^T\|\partial_t u^k\|^2_1dt\,\int_0^T\|u^k(t,.)\|^2_3 dt.
$$
This concludes the estimate of nonlinear divergence terms.\\

\noindent
{\it Step 4: Estimates of ``interior'' terms.} The interior terms are estimated with proposition \ref{int1} and lemma \ref{int2}. We first compute an estimate of $\|(A_0^{-T}-A_k^{-T})\nabla p^k\|_1$. One has
{\setlength\arraycolsep{1pt}
\begin{eqnarray}
\displaystyle
\|(A_0^{-T}-A_k^{-T})\nabla p^k\|_1&\leq&\|A_0^{-T}-A_k^{-T}\|_{\infty}\|p^k\|_2+\|\nabla(A_0^{-T}-A_k^{-T})\nabla p^k\|_0\nonumber\\
\displaystyle
&\leq&\|A_0^{-T}-A_k^{-T}\|_{\infty}\|p^k\|_2+\varepsilon^{\frac{1}{6}}\frac{\|\nabla(A_0^{-T}-A_k^{-T})\|_1}{\sqrt{\varepsilon}}\frac{\|p^k\|_2}{\sqrt{\varepsilon}}\nonumber\\
\displaystyle
&\leq&\frac{C(|\partial_x X^k-\partial_x X_0|_{\infty})}{\sqrt{\varepsilon}}\|X^k-X_0\|_3\|p^k\|_2.
\end{eqnarray}}

\medskip
As a consequence, we obtain the estimate
$$
\displaystyle
\int_0^T\|(A_0^{-T}-A_k^{-T})\nabla p^k\|_1^2\leq\frac{T}{\varepsilon}C(\sqrt{\frac{T}{\varepsilon}}\|u^k\|_{L^2((0, T), H^3)})\int_0^T\|u^k(t,.)\|_3^2dt\,\int_0^T\|p^k\|^2_2.
$$
\medskip
Similarly, the time derivative of this term satisfies
{\setlength\arraycolsep{1pt}
\begin{eqnarray}
\displaystyle
\int_0^T\|\partial_t\big((A_0^{-T}&-&A_k^{-T})\nabla p^k\big)\|_{(H^1)'}^2\nonumber\\
\displaystyle
&\leq& \frac{T}{\varepsilon}C\big(\sqrt{\frac{T}{\varepsilon}}\|u^k\|_{L^2((0, T),H^3)}\big)\int_0^T\|u^k(t,.)\|_3^2dt \int_0^T\|\partial_t p^k(t,.)\|_0^2dt\nonumber\\
\displaystyle
&+& \frac{T}{\varepsilon}C\big(\sqrt{\frac{T}{\varepsilon}}\|u^k\|_{L^2((0, T),H^3)}\big)\int_0^T\|\partial_t u^k(t,.)\|_1^2dt \int_0^T\|p^k(t,.)\|_2^2dt.\nonumber
\end{eqnarray}}
\noindent
The other nonlinear term with high derivatives that we consider is 
$$
\displaystyle
f^k=A_k^{-T}{\rm div}(A_k^{T}\mathcal{P}^k)-A_0^{-T}{\rm div}(A_0^T\mathcal{P}_0^k).
$$
\noindent
This is a lengthly but straightforward computation to show that
$$
\displaystyle
\hskip-6cm\int_0^T\|f^k\|_1^2+\|\partial_t f^k\|_{(H^1)'}^2 $$
$$\hskip4cm  \leq \frac{T}{\varepsilon}C\big(\sqrt{\frac{T}{\varepsilon}}\|u^k\|_{L^2((0, T), H^3)}\big)\int_0^T\|u^k\|_3^2\Big(\int_0^T\|u^k\|_3^2+\|\partial_t u^k\|_1^2\Big).
$$

\noindent {\it Step 5: estimates of boundary term at the free surface.} We have to deal with the nonlinear term
$$
\displaystyle
G^k_2=A_0^T\Big((\mathcal{P}^k-\mathcal{P}^k_0){\rm n}+(\mathcal{P}^k-p^k)({\rm n}_0-{\rm n})\Big)|_{z=\varepsilon}.
$$

On the one hand, one has
$$
\displaystyle
|(\mathcal{P}^k-\mathcal{P}_0^k){\rm n}|_{z=\varepsilon}|_{\frac{3}{2}}\leq C\big(\sqrt{\frac{T}{\varepsilon}}\|u^k\|_{L^2((0, T), H^3)}\big)|\mathcal{P}^k-\mathcal{P}_0^k|_{\frac{3}{2}}.
$$
\noindent
Moreover, this is a straightforward computation to show that
$$
\displaystyle
|\mathcal{P}^k-\mathcal{P}_0^k|_{\frac{3}{2}}\leq C\big(\frac{\|X^k-X_0\|_3}{\sqrt{\varepsilon}}\big)\frac{\|X^k-X_0\|_3\|u^k\|_3}{\varepsilon}.
$$
As a consequence, one obtains
$$
\displaystyle
\int_0^T|(\mathcal{P}^k-\mathcal{P}_0^k){\rm n}|_{z=\varepsilon}|_{\frac{3}{2}}^2\leq \frac{1}{\varepsilon^2}C\big(\sqrt{\frac{T}{\varepsilon}}\|u^k\|_{L^2((0, T), H^3)}\big)\big(\int_0^T\|u^k\|_3^2\big)^2.
$$

For the term $(\mathcal{P}^k-p^k)({\rm n}-{\rm n}_0)|_{z=\varepsilon}$, we first remark that $({\rm n}-{\rm n}_0)|_{z=0}=0$.  

\noindent
As a consequence, we obtain the estimate
$$
\displaystyle
\hskip-6cm \int_0^T|(\mathcal{P}^k-p^k)({\rm n}-{\rm n}_0)|_{z=\varepsilon}|_{\frac{3}{2}}^2 $$
$$ \hskip3cm
\leq \frac{1}{\varepsilon}C\big(\sqrt{\frac{T}{\varepsilon}}\|u^k\|_{L^2((0, T), H^3)}\big)\int_0^T\|u^k\|_3^2\big(\int_0^T\|u^k\|_3^2+\|p^k\|_2^2\big).
$$

Hence, using the estimates on the linear problem and using that remaining terms in the Navier-Stokes
type system \eqref{ns_approx} satisfied by the high order shallow water approximation $(u_a,p_a)$ satisfy
$$\int_0^T \Bigl(\|f_a(X_0,Z_0)\|^2_1 + \|\partial_t f_a(X_0,Z_0)\|^2_{(H^1)'} +
    |g^a_{2,2}(X_0,Z_0)|_{3/2}^2 + |\partial_t g^a_{2,2}(X_0,Z_0)|_{-1/2}^2\Bigr)$$
$$  + \frac{1}{\varepsilon} \int_0^T \Bigl(\|H_a(X_0,Z_0)\|_2^2 + \|\partial_t H_a(X_0,Z_0)\|_0^2 
+ |g_1^a(X_0,Z_0),g_{2,1}^a(X_0,Z_0)|^2_{3/2} $$
$$      + |g^a_{1,t}(X_0,Z_0),\partial_t g^a_{2,1}(X_0,Z_0)|_{-1/2}^2)\Bigr)
    \le C \varepsilon^ 4.$$
We have then proved the following proposition.
\begin{proposition}
Assume that $\|u_0\|_2\leq \frac{C-1}{C^2}\varepsilon^{\frac{3}{2}}$, then for $\varepsilon$ small enough, one has, for any $k\geq 1$
$$
\|(u^k,p^k)\|_X^2 \leq C(\|u_0\|^2_2+ \varepsilon^3).
$$
\end{proposition}
 At this step, there are two strategies: Either we use a Banach fixed point argument  and in this case we need estimates for higher Sobolev norms that are obtained in a way similar to our previous calculation. In this case, we can prove that the sequence $(u^k,p^k)$ is a Cauchy sequence respectively in the space ${\cal C}(0,T;H^{2+\eta}(\Omega))$ and $L^2((0, T), H^{2+\eta}(\Omega))$
with $0< \eta <1/2$, see \cite{Beale0} for further references. In this case, the result is not sharp. To obtain the well posedness in critical spaces, one has to use Tychonoff  fixed point procedure as in
\cite{CoSh}. There is no difficulty here to follow both strategies to obtain the well posedness result in our situation since the estimate in high norms is proved.

\section{Conclusion}

In this paper, we have proved the well-posedness of free surface Navier-Stokes equations with a slip condition at the bottom and convergence to a viscous shallow water model in the shallow water scaling $\varepsilon\to 0$, $\varepsilon$ being the aspect ratio. In contrast to our mathematical justification of a shallow water model for a fluid flowing down an inclined plane, we dropped the tension surface effect (which is important in order to deal with realistic situations). This is due to the fact that we worked in the more suitable Lagrangian coordinates. There is no difficulty to extend this result for fluids flowing down an inclined plane when the slope $\theta$ is asymptotically small, $\theta\sim\varepsilon$ and when the uniform flow is {\it stable} (see \cite{Bresch_Noble} for more details). 

When the uniform flow is unstable, it is known that instabilities so called roll-waves appear: these are periodic travelling waves. {\it Small amplitude} roll-waves are proved to exist both in viscous shallow water equations \cite{Needham} and free surface Navier-Stokes equations \cite{Nishida2}. However, the result of {\sc Nishida} et al. is proved in the presence of {\it surface tension} whereas roll-waves are proved to exist in viscous shallow water equations without surface tension and are of large amplitude when  the viscosity is small. In order to obtain a larger range of validity of shallow water equations, this would be of interest  to prove the existence of roll-waves in Navier-Stokes equations without surface tension that are close to ``shallow water roll-waves'' in the shallow water scaling.

\end{document}